\documentclass[a4paper,12pt,leqno]{amsart}
\makeatletter
\usepackage{amsfonts,delarray,amssymb,amsmath,amsthm,a4,a4wide}
\usepackage{graphicx}
\usepackage[ansinew]{inputenc}
\newtheorem{theo}{Theorem}
\newtheorem{lem}{Lemma}[section]
\newtheorem{defi}{Definition}
\newtheorem{pro}{Proposition}[section]

\newtheorem{coro}{Corollary}[section]

\newtheorem{remark}{Remark}[section]

\def\Xint#1{\mathchoice
   {\XXint\displaystyle\textstyle{#1}}%
   {\XXint\textstyle\scriptstyle{#1}}%
   {\XXint\scriptstyle\scriptscriptstyle{#1}}%
   {\XXint\scriptscriptstyle\scriptscriptstyle{#1}}%
   \!\int}
\def\XXint#1#2#3{{\setbox0=\hbox{$#1{#2#3}{\int}$}
     \vcenter{\hbox{$#2#3$}}\kern-.5\wd0}}
\def\dashint{\Xint-}

\include{defs}

\def\({\left(}
\def\){\right)}
\def\1{\mathbf{1}}

\def\D{\displaystyle}

\def\dt0{{{\frac{d}{dt}}_{|t=0}}}

\def\Et{{\widetilde E}}

\def\Eo{{E_\ep^{\om_0}}}
\def\ep{\varepsilon}
\def\E{{E_\ep^\om}}
\def\En{{E_{\ep_n}^\om}}
\def\eam{{F_\ep^\om}}
\def\FH{{\widehat{F}_\ep^\om}}

\def\form{{ f^{\om}_{R,m}}}
\def\formi{{ f^{\om}_{R,m_i}}}
\def\fb{{\underline f^{\om,x}}}

\def\fbe{{\underline f^{\om+\frac x\ep,x}}}
\def\fdm{{\mathcal{L}}}

\def\fe{{F_\ep}}

\def\fy{{f_\ep^{\om+y,x}}}
\def\F{{\mathcal F}}
\def\f{{f_\ep^{\om,x}}}

\def\gdm{{\mathcal{G}}}

\def\hal{\frac{1}{2}}

\def\indic{{\mathbf 1}}
\def\lag{{e_\ep^{\om,x}}}

\def\lip{\text{Lip}}

\def\loc{{\text{loc}}}
\def\lp{{L^2_{\text{loc}}}}
\def\Lp{{L^p_{\text{loc}}}}

\def\mt{{\tilde m}}
\def\mn{\mathbb{N}}
\def\mr{\mathbb{R}}
\def\mz{\mathbb{Z}}
\def\nab{\nabla}
\def\Om{\Omega}
\def\om{\omega}
\def\Pe{{P_\ep^\om}}
\def\Pey{{P_\ep^{\om+y}}}
\def\Pen{{P_{\ep_n}^\om}}
\def\Peny{{P_{\ep_n}^{\om+y}}}

\def\Poy{{P^{\om+y}}}
\def\Po{{P^\om}}
\def\Pox{{P^\om_x}}

\def\Qe{{Q_\ep^\om}}
\def\Qeb{{\underline Q_\ep^\om}}
\def\Qoy{{Q^{\om+y}}}
\def\Qo{{Q^\om}}
\def\sd{\bigtriangleup}
\def\sm{\setminus}
\def\tog{{\overset{\Gamma}{\to}}}
\def\ue{{u_\ep}}
\def\uen{{u_{\ep_n}}}
\def\ve{{v_\ep}}
\def\tve{{\tilde v_\ep}}
\def\tuei{{\tilde u_{\ep,i}}}
\def\tu{{\tilde u}}
\def\vp{\varphi}

\def\xe{{\frac x\ep}}

\def\txi{{\tilde x_i}}
\DeclareMathOperator{\bv}{BV}
\usepackage{color}

\title{A Two scale $\Gamma$-convergence  Approach for  Random Non-Convex Homogenization}
\author{Leonid Berlyand, Etienne Sandier and Sylvia Serfaty}
\address[Leonid Berlyand]{Pennsylvania State University\\
University Park, PA 16802, USA.}
\email{berlyand@math.psu.edu}

\address[Etienne Sandier]{Universit\'e Paris-Est\\
LAMA -- CNRS UMR 8050,\\
61, Avenue du G\'en\'eral de Gaulle, 94010 Cr\'eteil, France.}
\email{sandier@u-pec.fr}

\address[Sylvia Serfaty]{Sorbonne Universit\'es, UPMC Univ. Paris 06, CNRS, UMR 7598, Laboratoire Jacques-Louis Lions, 4, place Jussieu 75005, Paris, France.
 \newline \& Institut Universitaire de France \newline \& 
Courant Institute, New York University, 251 Mercer st, New York, NY 10012, USA.} 
\email{serfaty@ann.jussieu.fr}

\begin{document}
\begin{abstract}
We propose an abstract framework for the homogenization of random functionals which may contain  non-convex terms,  based on  a two-scale $\Gamma$-convergence approach and a definition of  Young measures on micropatterns which encodes the profiles of the oscillating functions and of functionals. Our abstract result is a lower bound for such energies in terms of a cell problem (on large expanding cells) and the $\Gamma$-limits of the functionals at the microscale. We show that our method allows to retrieve the results of Dal Maso and Modica in the well-known case of  the stochastic homogenization of convex Lagrangians. As an application, we also show how our method allows to stochastically homogenize a variational problem introduced and studied by Alberti and M\"uller,  which is a paradigm of a problem where an additional mesoscale arises naturally due to the non-convexity of the singular perturbation (lower order)  terms in the functional.
\end{abstract}
\maketitle

\noindent
{\bf keywords: }  stochastic homogenization, $\Gamma$-convergence, integral functionals
\\
{\bf MSC classification: }35B27, 60H25, 35J20

\section{Introduction}
The goal of this work is to develop an abstract framework for two-scale $\Gamma$-convergence of random non-convex functionals, and to show how this framework applies to specific problems in random homogenization. 
 
Random homogenization of convex functionals was studied in the seminal papers of Dal Maso and Modica \cite{dalmasopaper,dmm}, which introduced the fundamental idea of using sub-additivity combined with the (sub-additive) ergodic theorem, thus generalizing the homogenization of random linear elliptic equations (which correspond to  quadratic energies), an issue  which has attracted significant attention since the seminal papers \cite{kozlov,pv}.
For recent developments and state of the art  in the  homogenization of convex energy functionals, one may see  \cite{armstrongsmart} (which develops a quantitative approach) and \cite{duerinckxgloria} (which contains extensions to nonconvex and unbounded situations) and references therein.
Here we are rather interested in the homogenization of variational problems, in particular, allowing for non-convex lower order terms (possibly singular perturbations of the leading order terms) in the energy. 

It is well known in the mathematical theory of  homogenization that the non-convexity of the Lagrangian presents a major difficulty in applying the usual homogenization techniques to nonlinear problems  (see  e.g.~\cite{allaire,braidesf,muller87}).  We recall the issue in the simplest context of periodic homogenization. Then the computation of the effective coefficients   leads to a so-called expanding cell problem (whose size $R$ goes to infinity). For convex problems, this cell problem reduces to a problem on a single periodicity cell due to the uniqueness of the minimizers. However, for non-convex problems, non-uniqueness implies that  minimizers with multiple of the basic period may exist and such a reduction is no longer possible.  Moreover,
analysis of the oscillating test functions that are used in establishing convergence shows that a new {\it mesoscopic scale} may arise in the process of convexification (see, e.g.,  the work of M\"uller  \cite{muller87}, where the phenomenon of
a  mesoscale arising in minimizers of non-convex problems was demonstrated in a vectorial  elasticity  problem). This lies at the heart of the major computational challenges in the numerical analysis of such problems. The heuristic idea behind the concept of mesoscale can be explained as follows. Applying the $\Gamma$-convergence method one can see that the system may try to reach multiple local equilibrium positions due to non-uniqueness but the average
macroscopic gradient condition penalizes any deviation from the linear behavior. A balance between these two trends results in the convexification of the energy and the rise of a new mesoscale much larger than the microscopic period (cf. \cite{muller}). Note that the  homogenization of convex random (stationary and ergodic) problems results in a cell problem defined on the entire space and therefore  does not give rise to a mesoscale.

Let us now  recall briefly the  variational setting of convex nonlinear random homogenization problem as proposed  by Dal Maso-Modica in \cite{dmm,dalmasopaper}. Consider a functional 
\begin{equation}\label{dalmasosetting} 
F_\ep(v, G)= \dashint_G \fdm( x/\ep, \nabla v(x)) dx,
\end{equation}
 where $\fdm: \mathbb{R}^n \times \mathbb{R}^n \to \mathbb{R}_+ $ is a random Lagrangian which satisfies the usual growth and convexity  conditions.

The homogenization problem  can be stated as the question of determining the $\Gamma$-limit of this functional as $\ep \to 0$. Assuming stationarity and ergodicity,  it is proved in  \cite{dmm, dalmasopaper} that the $\Gamma$-limit of $F_\ep$ is an integral functional of the same type as  \eqref{dalmasosetting}, where the limiting (homogenized)   Lagrangian is  computed as the expectation of  a minimization problem for a local Lagrangian over an expanding  cell.   

For more general Lagrangians, for instance of the form $f_\ep(x, v, \nabla v)$, one can no longer expect the $\Gamma$-limit to be of the same integral form.  A classical example  due to Modica-Mortola \cite{mm}  features the convergence of  scalar soft interface models of the integral functional form to a sharp interface model that is no longer an integral functional. Another  striking example of this phenomenon was  studied by Müller in \cite{muller} and generalized by Alberti and Müller in  \cite{am}. In particular, in \cite{am} the important  notion of the Young measure on micropatterns was introduced for  two-scale 
$\Gamma$-convergence  and applied to the  following non-convex Lagrangian
\begin{equation}\label{am_Lagrangian} 
 L(x, v):=\ep^2 {v''}^2 +W(v') +a(x) v^2,
 \end{equation}
 where $W$ is the standard double-well potential $W(x) = (1-|x|^2)^2$.   A crucial feature of the Lagrangian $L(x,v)$ is the non-convexity of  the double well potential $W(t):=(1-t^2)^2$, which is mainly responsible for  the more complex form of the $\Gamma$-limit, the latter being  expressed  in terms of jumps of BV functions rather than by an integral functional. 

In this work we  extend, in  a way inspired by the abstract method in  \cite{ss1,ss2} (itself following a suggestion of Varadhan), the notion of Young measure on micropatterns introduced  in \cite{am}. While the standard Young measure encodes the frequency of taking certain values of  an oscillating sequence of functions,  the Young measure on micropatterns additionally encodes the profiles (or shape of the graph) of the oscillating functions at a given  scale. Note that this scale is determined by the problem.  In our work we  incorporate to the Young measure on micropatterns  the profiles of a sequence of  {\em oscillating functionals} that is, in the particular case of integral functionals, the profiles of the  oscillating coefficients of these functionals. For instance, in the case of a Lagrangian of the type $a_\ep(x)|\nabla u(x)|^2$ the profiles of  the oscillating functionals $\int a_\ep(x)|\nabla u(x)|^2$ reduce to the profiles of the oscillating functions $a_\ep$. 

However, using the natural action of translations on a function space and the idea from Dal Maso-Modica of metrizing the space of functionals endowed with the topology of $\Gamma$-convergence, defining the Young measure of oscillating functionals can be done in an  abstract setting without referring to a specific form of the functionals. This leads to what we consider to be the natural lower bound for variational problems where minimizers and  coefficients oscillate on the same scale.
This is of course completely natural for linear or more generally for convex problems, due to uniqueness. However here we consider non-convex nonlinear problems where minimizers can develop their own new scale of oscillations (see again \cite{ am}).

Loosely speaking, our main result (Theorem~\ref{th2}) may then be explained as follows: we are able to extend the lower-bound of Dal Maso-Modica to nonconvex local Lagrangians, whose coefficients may oscillate randomly. This lower bound is computed as follows: first we compute the $\Gamma$-limit of the local Lagrangian and then the desired lower bound is expressed in terms of a cell problem in the entire space for this $\Gamma$-limit. Note that the upper bound is computed by constructing appropriate test functions and  is usually  problem-specific, which is why we do not address it in a general abstract framework but  rather provide an example.

As a first application we  recover the theory of Dal Maso-Modica for random convex homogenization using our  abstract approach (Theorem~\ref{dalmasoth}).  Then we show how our framework applies to   a random version  of the one-dimensional model studied by Alberti-Müller in \cite{am}. In this second application, the techniques of Dal Maso-Modica would not apply due to non-convexity.   On the other hand the framework of \cite{am} developed for non-convex problems would not suffice because of the randomness --- or more precisely because of the oscillations of the functionals (brought in due to randomness) at the same scale as the oscillating functions (in fact, this direction is mentioned in \cite{am} at the end of \S~6.2 as ``particularly interesting"). Our approach is able to handle both randomness and non-convexity. Moreover,  once a matching upper bound is derived, it  provides the leading term for the asymptotics  of  the minimal energy for this model (Theorem~\ref{minam}).  
The  $\Gamma$-convergence itself, as is well-known, follows the same lines but can be considerably more technically involved. We believe that a number of problems with  functionals featuring oscillations either deterministic or random could be studied using our approach.

%Our main tool will be the notion of $\Gamma$-convergence, which amounts to deriving optimal lower and upper bounds for the energy.
%  The upper bound is computed by constructing appropriate test functions and  is usually  problem-specific. The difficult part is usually the lower bound, because it requires  to guess and justify rigorously the actual structure of minimizers. Our framework addresses precisely the lower bound aspect of the problem. For the specific example of the random version of \eqref{am_Lagrangian}, we will also show $\Gamma$-convergence on the level of  minimizers. The  $\Gamma$-convergence itself, as is well-known, follows the same lines but can be considerably more technically involved. We believe that a number of problems with  functionals featuring oscillations either deterministic or random could be studied using our approach.

The plan of the paper is the following. In Section~\ref{recast} we describe the  abstract setting in which the Young measure on micropatterns and functionals is defined. Here we also  give the abstract version of the $\Gamma$-convergence lower bound in this setting. In Section~\ref{convex} we show how the theory of Dal Maso-Modica  \cite{dmm, dalmaso}  can be recovered in our framework. Finally in Section~\ref{nonconvex} we analyze the $\ep\to 0$ limit of a generalization of the Alberti-Müller functional  \eqref{am_Lagrangian} to a random setting by applying our approach. 

%%To achieve this we use  the notion of $\overline\Gamma$-convergence introduced in \cite{dmm}  (see \cite{dalmaso} for a systematic exposition), in a simplified form  suitable for the description of variational convergence of random stationary functionals.
%%We present an abstract framework for the homogenization of random stationary non-convex energy functionals. Our approach is related to (add references). cell problem in whole space, no additive ergodic theorem (due to non-convexity?) no boundary condition. 

%%\textcolor{red}{A key feature of our approach is that unlike the in the previous  approaches \cite{jko} our associated cell problem is defined on the whole space with boundary conditions at infinity (rather than taking a limit of expanding domains).This cell problem is introduced  due  to the fact that  non-convexity does not allow the use of subadditive ergodic theorem. Sylvia is this actually true?} 

\section{Abstract Setting}

\subsection{$\Gamma$-convergence of extended functionals}\label{recast}

Hereafter, $(X,d_X)$ is a Polish space which should be thought of as a function space to which the unknown function $u$ belongs, $\Om$ is a probability space whose generic element is denoted by $\om$ and on which the probability measure is simply denoted by $d\om$. 

The space $\mr^n$ acts on $\Om$ by measurable isomorphisms, this action is denoted by $(\om,y)\to\om+y$. It also acts on $X$ and this  is  denoted by $(u,y)\to \theta_y u$, where $u\in X$ and $y\in\mr^n$ ($\theta_y u$ should be thought of as $u(\cdot+y)$ if $X$ is a space of functions defined on $\mr^n$). We assume this action is continuous in $u$ uniformly with respect to $y\in K$ for any compact subset $K$ of $\mr^n$, and  also continuous in $y$. 

As mentioned above we wish to relax the convexity assumption of \cite{dmm, dalmasopaper} and allow for lower-order terms. This requires the introduction of a more general class than integral functionals, which will be closed with respect to $\Gamma$-convergence. For this we replace the notion of Lagrangian with the notion of {\em integrand} --- which is a functional ---  analogous to the notion introduced in \cite{am}.

An {\em integrand} is a map from $X\times\mr^n\to \mr_+\cup\{+\infty\}$, denoted by $f:(u,y)\to f(u,y)$. To see how this relates to \eqref{dalmasosetting}  let us write the functional \eqref{dalmasosetting} in terms of the rescaled function $u(y) = \ep^{-1} v(\ep y)$ as 
\begin{equation}\label{4} F_\ep(v,G) = \dashint_{G} \fdm(x/\ep,\nabla u(x/\ep))\,dx = \dashint_{G/\ep} \fdm(y,\nabla u(y))\,dy.
\end{equation}
Next we approximate $\fdm(x/\ep,\nabla u(x/\ep))$ by its  local average: $$
F_\ep(v,G) \approx   \dashint_{G}\dashint_{B(x/\ep,1)} \fdm(z,\nabla u(z))\,dz\,dx.
$$
The approximation property easily follows from Fubini's theorem. Now we define an integrand $f$   as follows
\begin{equation}\label{5} f(u, y) = \dashint_{B(y,1)} \fdm(z,\nabla u(z))\,dz,
\end{equation}
and we find 
\begin{equation}\label{6}
F_\ep(v,G) \approx   \dashint_{G} f(u, x/\ep)\,dx.
\end{equation}
This procedure was already present in \cite{am}, and used in the abstract method of \cite{ss1,ss2}.
Note that we have now replaced the Lagrangian, which is a function on $\mr^n\times\mr^n$ by an integrand, which is a functional on $X\times \mr^n$, where $X$ is a function space. This integrand is  defined as the local average of the Lagrangian. This rewriting allows us to consider very general problems, e.g. non-convex, and also to establish a two-step $\Gamma$-convergence procedure as in \cite{am}, where first the $\Gamma$-limit of the integrand (which is trivial in the case of \eqref{5}) is computed, followed by computing the $\Gamma$-limit of the full energy. 

We define the topological structure on our space of functionals following \cite{dmm, dalmasopaper} (see also the book by Dal Maso \cite{dalmaso}). We will say that a sequence of functionals $\{f_n:X\times\mr^n\to \overline\mr_+\}_n$ $\Gamma$-converges to a functional $f$ if for any $(u,y)\in X\times \mr^n$ we have 
$$\lim_{\ep\to 0} \limsup_{n\to +\infty} \inf_{B_\ep(u,y)} f_n = \lim_{\ep\to 0} \liminf_{n\to +\infty} \inf_{B_\ep(u,y)} f_n = f(u,y).$$
In our case, where the topology of $X\times\mr^n$ has a countable basis, this is equivalent to the following sequential characterization.

\begin{itemize}
\item[i)] For any convergent sequence $(u_n,y_n)\to (u,y)$, it holds that $\liminf_n f_n(u_n,y_n)\ge f(u,y).$
\item[ii)] For any $(u,y)$ there exists a sequence $(u_n,y_n)$ converging to $(u,y)$ such that \\ $\limsup_n f_n(u_n,y_n)\le f(u,y)$.
\end{itemize}
This is the usual definition of $\Gamma$-convergence, {\em with respect to the couple} $(u,y)$. 

Note that the topology of $\Gamma$-convergence is not even separated. However, when restricted to suitable sets of functionals, it becomes metrizable (see \cite{dalmaso} and the sketch below).

The topology of $\Gamma$-convergence is compatible with the action of $\mr^n$ on $X\times\mr^n$, defined by 
\begin{equation}\label{act}\theta_y f(u,z)= f(\theta_{-y} u, z+y).\end{equation}
\begin{lem}\label{gammatrans}
If $f$ is a lower semicontinuous functional, then $y\mapsto \theta_y f$ is continous.
\end{lem}
\begin{proof} Clearly it suffices to prove continuity at $y=0$. Thus we consider a sequence $y_n$ converging to zero and prove that $\theta_{y_n} f$ $\Gamma$-converges to $f$.

For the lower bound part, assume $(u_n,z_n)$ converges to $(u,z)$. Then, from  definition \eqref{act}, $\theta_{y_n} f(u_n,z_n) = f\(\theta_{-y_n} u_n, z_n+y_n\)$. From the uniform continuity of $\theta_y u$ with respect to $u$ and the continuity with respect to $y$ we have $\theta_{-y_n} u_n \to u$ while $z_n+y_n \to z$, hence from the lower semicontinuity of $f$ we find 
$$ \liminf_n \theta_{y_n} f(u_n,z_n)\ge f(u,z).$$
For the upper bound part, assume $(u,z)\in X\times \mr^n$. Then $u_n:= \theta_{y_n} u$ converges to $u$, $z_n = z - y_n$ converges to $z$, and for any $n$ we have 
$$ f(u,z) =  \theta_{y_n} f(u_n,z_n).$$
\end{proof}

This topology is compact and  metrizable on suitable sets of functionals. First we define the Yosida regularization of $f$ with parameter $\lambda$ to be the functional
\begin{equation}\label{yosida}R_\lambda f (u,y) := \inf_{(v,z)\in X\times \mr^n} g(v,z),\quad\text{where}\quad g(v,z) = f(v,z) + \lambda\(d_X(u,v) + |y-z|\).\end{equation}

Then we say a functional $f$ is {\em coercive} if  for any function $C:(0,+\infty)\to (0,+\infty)$, the set $\{u\in X\mid \forall R>0,\ \dashint_{B(0,R)} f(u,y)\,dy\le C(R)\}$ is compact in $X$.

Then we have (see \cite[Theorem~17.14]{dalmaso}):
\begin{pro} \label{compact} Let $\F$ be a family of lower semicontinuous functionals $f$ which are bounded below by a coercive lower semicontinuous functional $f_0$. Assume that for any $M, \lambda>0$, any $u\in X$ and any $y\in\mr^n$ there exists a  compact subset $K_{M,u,y,\lambda}$ of $X\times \mr^n$, such that if $R_\lambda f (u,y)\leq M$, then  the infimum defining $R_\lambda f (u,y)$ is in fact the same as the infimum on $K_{M,u,y,\lambda}$.

Then the topology of $\Gamma$-convergence is metrizable on $\F$, and $\F$ is compact.
\end{pro}

This motivates the following

\begin{defi} \label{defiM} We say that a family of functionals  has the (M) property  if there exist a coercive lower semicontinuous functional $f_0$ and a family of compact subsets $K_{M,u,y,\lambda}$ of $X\times \mr^n$,  such that for any $f$ in the family, $f_0\le f$ and  $R_\lambda f (u,y)<M$ implies that   the infimum defining $R_\lambda f (u,y)$ is in fact the same as the infimum on $K_{M,u,y,\lambda}$.\end{defi}

The proof of Proposition~\ref{compact} relies on two lemmas.

\begin{lem} Assume $\F$ satisfies the hypotheses of Proposition~\ref{compact}.  If $\{f_n\}$ is a sequence in $\F$ which $\Gamma$-converges to $f$, then the Yosida regularizations $R_\lambda f_n$ converge pointwise to $R_\lambda f$ for any $\lambda>0$.
\end{lem}
\begin{proof} The proof follows \cite{dalmaso}, Theorem~17.14. Assume $f_n\tog f$, and $(u_0,y_0)\in X\times\mr^n$. Let $f_{n,\lambda}(u,y)  = f_n(u,y) + \lambda\(d_X(u,u_0) + |y-y_0|\)$. Since $f_{n,\lambda} - f_n$ is continuous we have $f_{n,\lambda}\tog f_\lambda$, where $f_\lambda$ is defined in the obvious way. We now prove that this implies that the infimum of $f_{n,\lambda}$ converges to the infimum of $f_\lambda$, i.e. that $R_\lambda f_n(u_0,y_0)\to R_\lambda f(u_0,y_0)$, which will prove the lemma. 

If $\inf f_{n,\lambda}\to +\infty$, then $\inf f_\lambda=+\infty$ by using a recovery sequence. Thus we assume, after extracting a subsequence,  that $\lim_n \inf f_{n,\lambda}=\ell\in\mr$, and we prove that $\ell = \inf f_\lambda$, which will prove the convergence of the whole sequence. 

There exists $(u_n,y_n)$ such that $f_{n,\lambda}(u_n,y_n)\le \ell + o(1)$ as $n\to +\infty$, and from the hypotheses we may choose $(u_n,y_n)$ in the compact set $ K_{M,u_0,y_0,\lambda}$, where $M=\ell+1$. Then a subsequence (not relabeled) converges to $(u,y)$ and since $f_{n,\lambda}\tog f_\lambda$ we deduce $f_\lambda(u,y)\le \ell$. The fact that $f_\lambda(u,y)\ge \ell$ is clear by  using a recovery sequence.
\end{proof}

\begin{lem} Assume $\F$ satisfies the hypothesis of Proposition~\ref{compact}. If $\{f_n\}$ is a sequence in $\F$ such that  $R_\lambda f_n$ converges pointwise to $R_\lambda f$ for any $\lambda\in\mn^*$, then  $\{f_n\}$  $\Gamma$-converges to $f$.
\end{lem}
\begin{proof} Assume $R_\lambda f_n$ converges pointwise to $R_\lambda f$ for any $\lambda\in\mn^*$. Then for any sequence $(u_n,y_n)\to (u,y)$ we have 
$$f_n(u_n,y_n) +\lambda\(d_X(u_n,u) + |y_n-y|\)\ge R_\lambda f_n(u,y)\to R_\lambda f(u,y).$$
It follows that $\liminf_n f_n(u_n,y_n)\ge R_\lambda f(u,y)$ for any $\lambda\in\mn^*$. But $f$ is lower semicontinuous, therefore $f(u,y) = \sup_{\lambda\in\mn^*} R_\lambda f(u,y)$ and the lower bound part of $f_n\tog f$ follows.

For the recovery sequence part assume $(u,y)\in X\times \mr^n$. Then, since $R_\lambda f_n(u,y)\to R_\lambda f(u,y)$ and since $R_\lambda f(u,y) \to f(u,y)$ as $\lambda\to +\infty$, there is a subsequence $\lambda_n$ of integers tending to $+\infty$ such that $R_{\lambda_n} f_n(u,y) \to f(u,y)$. For each $n$ we may therefore find $(u_n,y_n)$ such that $R_{\lambda_n} f_n(u,y) \le f(u,y)+ o(1)$, which from the definition of $R_\lambda f$ clearly implies that $\limsup_n f_n(u,n,y_n)\le f(u,y)$, and that $(u_n,y_n)\to (u,y)$.
\end{proof}

\begin{proof}[Proof of the Proposition] Consider a dense sequence $(u_k,y_k)_{k\in\mn^*}$ in $X\times \mr^n$. Since $R_\lambda f$ is $\lambda$-Lipschitz for any $f$, for any $\lambda>0$ the pointwise convergence of $R_\lambda f_n$ to $f$ is equivalent to the convergence of $R_\lambda f_n(u_k,y_k)$ to $R_\lambda f(u_k, y_k)$ for any $k\in\mn^*$. Then we define 
$$d(f,g) = \sum_{\lambda,k\in\mn^*} 2^{-\lambda-k}\left|\frac{R_\lambda f(u_k,y_k)}{1+R_\lambda f(u_k,y_k)} - \frac{R_\lambda g(u_k,y_k)}{1+R_\lambda g(u_k,y_k)}\right|.$$
Then if $\{f_n\}$ is a sequence in $\F$, and $f\in\F$,  we have that $d(f_n,f)\to 0$ is equivalent to $R_\lambda f_n(u_k,y_k)\to R_\lambda f(u_k,y_k)$ for any $\lambda,k\in\mn^*$ which is equivalent to the $\Gamma$-convergence of $f_n$ to $f$, using the two previous lemmas. It is clear that $d$ is a distance on $\F$.

It remains to prove that $\F$ is compact for the topology of $\Gamma$-convergence. This is Theorem 8.5 in \cite{dalmaso}, which applies here without modification.
\end{proof}

\subsection{Stationarity}
Recall that in \cite{dmm, dalmasopaper}, the Lagrangian $\fdm$ is random although the random parameter does not appear in the notation. From now on we will make it appear explicitly, and also allow for a dependence both on the slow variable $x$ and on $\ep$, which is required to deal with the case studied in \cite{am}. Thus, we will consider a family of integrands $\{\f\}$ indexed by a positive number $\ep$, an element $\om$ in a probability space $\Omega$, and a variable $x$ belonging to a smooth bounded open subset $G\subset \mr^n$. A family of functionals $\{f^\om\}$ indexed by the random parameter is said to be stationary if $\theta_y f^\om = f^{\om+y}$, where the action  $\theta$ is as in \eqref{act}. In terms of our family of functionals $\{\f\}$, this means that for any $x\in G$, $\om\in\Om$, $\ep>0$ and $y\in\mr^n$ we have
\begin{equation}\label{invariance}\theta_y\f = \fy.\end{equation}
Note that  $\ep$ and $x$ are fixed in this definition. 

Going back to problem \eqref{dalmasosetting}, and introducing explicitly the parameter $\om$, we let
\begin{equation}\label{dict1} \f(u,y') = \dashint_{B(y',1)} \fdm(\om, z, \nabla u(z))\,dz.\end{equation}
Then 
\begin{equation}\label{dict2} \theta_{y}\f(u,y') = \dashint_{B(y'+y,1)} \fdm(\om, z, \nabla u(z-y))\,dz = \dashint_{B(y',1)} \fdm(\om, z+y, \nabla u(z))\,dz.\end{equation}
Therefore stationarity in this case means that $\fdm(\om, z+y ,p) = \fdm(\om+y,z,p)$. Note that the 
the fact that for any $y$ the map $\om\to \om+y$ is measure preserving and such that $\fdm(\om, z+y ,p) = \fdm(\om+y,z,p)$ implies that $\fdm(\cdot, z+y,p)$ and $\fdm(\cdot,y,p)$ have the same distribution, which is the hypothesis made in \cite{dalmasopaper}. 

From now on we assume that the family $\{\f\}$ is stationary and has the (M) property as defined above, hence from Proposition~\ref{compact} is included in a family $\F$ which is compact metrizable, and in particular is a Polish space. The distance function on $\F$ will be denoted by  $d$. 

\subsection{Probability measure on profiles and functionals}

Given a family $\{\f\}$ as above, we define the random functional $\E$ by letting, for any $u\in X$, 
\begin{equation}\label{energy}\E(u) = \dashint_G \f\(u,\frac x\ep\)\,dx.\end{equation}

To see that \eqref{energy} makes sense we first note that, from the very definition of $\Gamma$-convergence we have
\begin{lem} \label{sci} The map $(y,f,u)\mapsto f(u,y)$ is lower semi-continuous on $\mr^n\times\F\times X$. 
\end{lem}
Given a fixed $u\in X$, the map $x\mapsto \f\(u,\frac x\ep\)$ is the composition of $x\mapsto \(\frac x\ep, \f, u\)$ --- which is Borel measurable since each component is ---  and $(y,f,u)\mapsto f(u,y)$ which is lower semicontinuous. Thus it is Borel measurable and positive, and \eqref{energy} makes sense.

We may rewrite the energy $\E$ in a way which is convenient to take limits.

\begin{defi} Assume $\{\ue\}_\ep$ is a family in $X$. For any $\om$ we define $\Pe$ to be the image of the normalized Lebesgue measure on $G$ by the map $x\mapsto\(x,\theta_\xe\f,\theta_\xe\ue\)$. We write this 
\begin{equation}\label{pe}\Pe = \dashint_G \delta_{\(x,\theta_\xe\f,\theta_\xe\ue\)}.\end{equation}
Then $\Pe$ is a Borel measure on $G\times\F\times X$. 
\end{defi}

To prove that this definition actually defines a Borel measure, it suffices to show that  $x\mapsto\(x,\theta_\xe\f,\theta_\xe\ue\)$ is Borel measurable. This is obvious for the components $x$ and 
$\theta_\xe\ue$ which are continuous with respect to $x$. For the last component we have 

\begin{lem} The map $x\mapsto \theta_\xe\f$ is Borel measurable.\end{lem}
\begin{proof} We write this map as the composition of $x\mapsto (x,\xe)$ and $(x,y)\mapsto \theta_y\f$, hence it suffices to show that this last map is Borel measurable. It is Borel measurable with respect to $x$ and therefore it will be Borel measurable if we prove it is continuous with respect to $y$ \cite{gow}. This is the content of Lemma~\ref{gammatrans}.
\end{proof}

We may now rewrite $\E$ as follows
\begin{pro} With the above notation,
\begin{equation}\label{ebis} \E(\ue) = \int \Phi(f,u)\,d\Pe(x,f,u),\end{equation}
where 
$$\Phi(f,u):=f(u,0).$$
\end{pro}
Note that $\Phi$ is lower semicontinuous from Lemma~\ref{sci} hence the integral in \eqref{ebis} makes sense. The one line proof of \eqref{ebis} is 
\begin{multline*} \E(\ue) = \dashint_G \f\(\ue,\xe\)\,dx =  \dashint_G \theta_\xe \f\(\theta_\xe\ue,0\)\,dx \\=  \dashint_G \Phi\(\theta_\xe \f,\theta_\xe\ue\)\,dx = \int \Phi(f,u)\,d\Pe(x,f,u).\end{multline*}

For the integral \eqref{energy} to make sense, it suffices that  $\f$ be Borel measurable with respect to $x$. We actually make the stronger assumption that it is {\em uniformly measurable} with respect to $x$ in the following sense.

\begin{defi} The family  $\{\f\}$  is {\em uniformly measurable} with respect to $x$ if for any $\delta>0$, 
\begin{equation}\label{um} \lim_{h\to 0}\left|\{x\in G\mid \sup_{\om} d(\f, \theta_y f_\ep^{\om, x+h}) >\delta\}\right| = 0,\end{equation}
and the limit is uniform with respect to $\ep$. 
\end{defi}

\subsection{Passing to the limit}
We now wish to compute a $\Gamma$-convergence lower bound for $\E$. The first step is  to pass to the limit in \eqref{ebis}  in terms of the probability measure $\Pe$ given in  \eqref{pe}. The  limiting measure $\Po= \lim \Pe$ is precisely what is referred to in the introduction as the Young measure on micropatterns and functionals.

\begin{pro}\label{tight} Assume that $\E(\ue) \le C$ for every $\ep>0$. Then for any sequence $\{\ep_n\}_n$ converging to $0$,  the family of probability measures  $\{\Pen\}_n$ is relatively compact, and any accumulation point $\Po$ is a probability measure on $G\times\F\times X$.
\end{pro}
\begin{proof} Since $\{\Pen\}_n$ are probability measures on a Polish space, we must prove that the family is tight. We use a criterion for tightness whose proof is due to E. Lesigne (private communication, see \cite[Lemma~2.1]{ss1}):  $\{\Pen\}_n$ is tight if and only if for any $\delta>0$ and any $n$ there exists $K_{n,\delta}\subset G\times\F\times X$ (which {\em need not} be compact) such that (i) $\Pen(K_{n,\delta})>1-\delta$ and (ii) if $(x_n, f_n, u_n)\in K_{n,\delta}$ for every $n$, then $\{(x_n, f_n, u_n)\}_n$ is relatively compact. Note that for any given $\delta$, (i) need only be satisfied for $n$ large enough, since we may replace $K_{n,\delta}$ by $G\times\F\times X$ for $n<n_0$ without altering condition (ii).

We let, recalling that $f_0$ is a coercive lower semicontinuous functional which bounds from below every $f\in\F$, 
$$K_\ep= \left\{u\in X \mid \forall 1\le k\le k_\ep,\ \dashint_{B_{2^k}} f_0(u,y)\,dy\le A 2^k\right\},$$
where $k_\ep$ is chosen such that, as $\ep\to 0$, 
$$k_\ep\to +\infty,\quad \ep 2^{k_\ep} \to 0.$$
We will prove that for any  sequence $\{\ep_n\}_n$ converging to $0$ and any $\delta>0$, if we choose $A$ large enough then $K_{n,\delta} := G\times\F\times K_{\ep_n}$ satisfies  (i) and (ii) above, hence proving the proposition. 

To lighten notation, we denote by  $\ep$ an element of the sequence  $\ep_n$, and write $\lim_{\ep\to 0}$ instead of $\lim_{n\to +\infty}$.

Let $G_\ep$ be the set of $x\in G$ such that $B(x,\ep 2^{k_\ep})\subset G$. Then, for every $1\le k\le k_\ep$ we have 
\begin{equation}\begin{split} C&\ge \int_G \f(\ue,x)\,dx \ge \int_{G_\ep}  \dashint_{B(x,2^k\ep)} \f(\ue,x'/\ep)\,dx'\,dx\\
                        &= \int_{G_\ep} \dashint_{B(0,2^k\ep)} f_\ep^{\om,x+x'}\(\ue,\frac{x+x'}\ep\)\,dx'\,dx\\
                        &= \int_{G_\ep} \dashint_{B(0,2^k)} f_\ep^{\om+\xe,x+y}\(\theta_\xe\ue,y\)\,dy\,dx\\
                        &\ge \int_{G_\ep} \dashint_{B(0,2^k)} f_0\(\theta_\xe\ue,y\)\,dy\,dx.\\
                        \end{split}\end{equation}
It follows that 
$$\left|\left\{x\in G_\ep \mid \dashint_{B(0,2^k)} f_0\(\theta_\xe\ue,y\)\,dy>A 2^k\right\}\right| \le \frac CA2^{-k},$$
and therefore, that 
$$\Pe\(G\times\F\times K_\ep\) \ge 1 - \frac 1{|G|}\(\sum_{k=1}^{k_\ep} \frac C{2^k A} + |G\setminus G_\ep|\right).$$
 From our choice of $k_\ep$, we have that $|G\setminus G_\ep|$ tends to $0$ as $\ep\to 0$, hence for any $\delta>0$, choosing $A$ large enough we find that $\Pe\(G\times\F\times K_\ep\) \ge 1 - \delta$ for any small enough $\ep$. This proves (i). 

Now we prove (ii). Assume $(\xe,f_\ep,\ue)\in G\times\F\times K_\ep$ for any $\ep$ belonging to the sequence $\{\ep_n\}$ which converges to $0$. Then there exists a subsequence, which we do not relabel, such that $\xe$ and $\fe$ both converge, since $G$ and $\F$ are relatively compact. Moreover, $\ue\in K_\ep$ implies that for any integer $k>0$ we have 
$$\limsup_{\ep\to 0} \dashint_{B_{2^k}} f_0(\ue,y)\,dy\le A 2^k,$$
hence for a further subsequence $\{\ue\}$ converges as well, using the coercivity of $f_0$. 
\end{proof}

We also have, using \eqref{ebis}, the fact that $\Phi$ is lower semicontinuous and Lemma~2.2 in \cite{ss2}, the following
\begin{pro}\label{lb1}  If $\E(\ue) \le C$ for every $\ep>0$, and if $\{\ep_n\}_n$ is a sequence converging to $0$  such that $\Pen\to\Po$ as $n\to +\infty$, then
$$\liminf_{n\to +\infty} \En(\uen) \ge \int f(u,0)\,d\Po(x,f,u).$$
\end{pro}

The next step consists in studying some properties of $\Po$.

\subsection{Invariance properties of $\Po$}

We may define an action on $G\times\F\times X$ by letting, for any $y\in\mr^n$, 
\begin{equation}\label{action} \theta_y(x,f,u) = (x,\theta_y f, \theta_y u).\end{equation}

Then we have
\begin{pro} \label{inv1} Assume that $\E(\ue)\le C$ for any $\ep>0$, and let $\{\ep_n\}_n$ be a sequence tending to $0$ such that $\Pen\to \Po$ as $n\to +\infty$, where $\Pe$ is defined in \eqref{pe}. 

Then $\Po$ is invariant under the action of $\theta$.
\end{pro}
\begin{proof} In the course of the proof and to lighten notation, we denote by $\ep$ a generic element of the sequence $\{\ep_n\}_n$ and write $\lim_{\ep\to 0}$ instead of $\lim_{n\to +\infty}$.

We need to prove that for any bounded continuous function $\vp$ on $G\times\F\times X$ and for any $y\in\mr^n$ we have 
\begin{equation}\label{eqinv}\int \vp(x,f,u)\,d\Po(x,f,u) = \int \vp(x,\theta_y f,\theta_y u)\,d\Po(x,f,u).\end{equation}
It is well known that it suffices to consider Lipschitz continuous functions. 

Consider then a bounded Lipschitz continuous function $\vp$ and a sequence $\{\ep_n\}_n$ converging to $0$ such that $\Pe\to\Po$. Then from the definition of $\Pe$, 
\begin{equation}\label{invsplit}\begin{split} 
\int \vp(x,\theta_y f,\theta_y u)\,d\Po(x,f,u) &= \lim_{\ep\to 0} \dashint_G \vp\(x,\theta_{\xe+y} \f, \theta_{\xe+y} \ue\)\,dx \\
								   &= \lim_{\ep\to 0} \dashint_{G+\ep y} \vp\(x-\ep y,\theta_{\xe} f_\ep^{\om, x-\ep y}, \theta_{\xe} \ue\)\,dx \\
								   &= \lim_{\ep\to 0} \dashint_{G} \vp\(x-\ep y,\theta_{\xe} f_\ep^{\om, x-\ep y}, \theta_{\xe} \ue\)\,dx, 
                        \end{split}\end{equation}
since $|G\triangle (G+\ep y)|$ tends to $0$ as $\ep\to 0$, where $\triangle$ denotes the symmetric difference of sets. Now, because of the uniform measurability \eqref{um},  for any $\delta>0$ we have that the measure of the set of $x\in G$ such that $d(\theta_{\xe} f_\ep^{\om, x}, \theta_{\xe} f_\ep^{\om, x-\ep y})>\delta$ tends to $0$ as $\ep\to 0$. Therefore,  as $\ep\to 0$, 
$$\left|\dashint_{G} \vp\(x-\ep y,\theta_{\xe} f_\ep^{\om, x-\ep y}, \theta_{\xe} \ue\) -  \vp\(x, \theta_{\xe} f_\ep^{\om, x}, \theta_{\xe} \ue\)\,dx\right|\le |G|(\delta +\ep |y|) \|\vp\|_{\lip} + o(1) |\vp|_\infty.$$
Since this is true for any $\delta>0$, we deduce that 
\begin{multline*}\lim_{\ep\to 0} \dashint_{G} \vp\(x-\ep y,\theta_{\xe} f_\ep^{\om, x-\ep y}, \theta_{\xe} \ue\)\,dx =\\ \lim_{\ep\to 0} \dashint_{G} \vp\(x,\theta_{\xe}\f, \theta_{\xe} \ue\)\,dx = \lim_{\ep\to 0} \int \vp\,d\Pe = \int \vp\,d\Po.\end{multline*}
Together with \eqref{invsplit}, this proves \eqref{eqinv}.
\end{proof}

Another useful invariance property of $\Po$ is 
\begin{pro}  Assume that $\E(\ue)\le C$ for any $\ep>0$, and let $\{\ep_n\}_n$ be a sequence tending to $0$ such that $\Pen\to \Po$ as $n\to +\infty$, where $\Pe$ is defined in \eqref{pe}. 

Then, for any $y\in\mr^n$, we have $\Peny\to\Poy$ as $n\to+\infty$, where  $\Poy$ is  the push-forward of $\Po$ by the map $(x,f,u)\to(x,f,\theta_{-y} u)$. In particular, denoting  by $\Qo$  the marginal of $\Po$ with respect to the first two variables, we have  $\Qoy = \Qo$.
\end{pro}
\begin{proof} Consider a sequence $\{\ep\}$ converging to $0$ and such that $\Pe\to\Po$. As in the proof of the previous proposition and to lighten notation, we denote by $\ep$ a generic element of the sequence $\{\ep_n\}_n$ and write $\lim_{\ep\to 0}$ instead of $\lim_{n\to +\infty}$ in the rest of the proof.

For any $y\in\mr^n$, the push-forward $\Poy$ of $\Po$ by the map $(x,f,u)\to(x,f,\theta_{-y} u)$ is the limit of the push-forward of $\Pey$ by the same map as $\ep\to 0$. Thus, considering as in the proof of Proposition~\ref{inv1}  a bounded Lipschitz continuous function $\vp$ on $G\times\F\times X$ we have 
$$\int \vp(x,f,u)\,d\Poy(x,f,u) = \lim_{\ep\to 0} \dashint_G \vp\(x, \theta_\xe \f, \theta_{-y}(\theta_\xe\ue)\)\,dx.$$
As in the previous proposition, this is inturn equal to 
$$\lim_{\ep\to 0} \dashint_G \vp\(x, \theta_{\xe+y} \f, \theta_{-y}(\theta_{\xe+y}\ue)\)\,dx,$$
and since $\theta_{\xe+y} \f = f_\ep^{\om+y,x}$ we obtain
$$\int \vp(x,f,u)\,d\Poy(x,f,u) =\lim_{\ep\to 0} \dashint_G \vp\(x, \theta_{\xe} f_\ep^{\om+y,x}, \theta_{\xe}\ue)\)\,dx,$$
which proves precisely that $\Pey\to\Poy$.
\end{proof}

\subsection{Lower bounds}
We now reformulate the lower bound from Proposition~\ref{lb1}  with the help of the ergodic theorem and the invariance properties it implies for the Young measure $\Po$. Recall that a family $\{U_R\}_{R>0}$  of subsets of $\mr^n$ is a Vitali family (see \cite{nmriv}) if  (i) the intersection of their closures is $\{0\}$,   (ii)  $R\mapsto|U_R|$ is left continuous, and (iii) $|U_R - U_R|\le C|U_R|$ for some constant $C>0$ independent of $R$.

\begin{pro}\label{th1} Given $\om\in\Omega$, assume $\{\f\}_{x,\ep}$ is a family of nonnegative lower semicontinuous random  functionals satisfying the invariance property \eqref{invariance}, which is uniformly measurable with respect to $x$ and bounded below by a lower semicontinuous coercive functional $f_0$. Also assume $\{\f\}$ has the (M) property. Assume that $\E(\ue)\le C$ for any $\ep>0$, and let $\{\ep_n\}_n$ be a sequence tending to $0$ such that $\Pen\to \Po$ as $n\to +\infty$, where $\Pe$ is defined in \eqref{pe}. 

Then, denoting $Q^\om$ the marginal of $P^\om$ with respect to the first two variables,  we have 
\begin{equation}\label{binf}\begin{split}
\liminf_{n\to +\infty} \En(\uen) &\ge \int\(\limsup_{R\to+\infty} \dashint_{B_R} f(u,y)\,dy\) \,d\Po(x,f,u)\\
&\ge\int \inf_{u\in X} \(\limsup_{R\to+\infty} \dashint_{B_R} f(u,y)\,dy\) \,d\Qo(x,f).
\end{split}\end{equation}
Moreover, in the above statement, the family of balls $\{B_R\}_R$ may be replaced by any Vitali family of bounded open sets $\{U_R\}_{R>0}$ such that,  
\begin{equation}\label{hypsets}  
\forall\lambda\in\mr^n,\quad  \lim_{R\to +\infty} \frac{|(\lambda + U_R)\sd U_R
|}{|U_R|} = 0. \end{equation}
\end{pro}

\begin{proof} The first inequality in \eqref{binf} is deduced from Proposition~\ref{lb1} and the invariance of $\Po$ under $\theta_y$. Indeed this invariance implies, using Wiener's multiparameter ergodic theorem (see \cite{becker}),  that 
\begin{multline*}\int f(u,0) \,d\Po(x,f,u) = \int\(\limsup_{R\to+\infty} \dashint_{B_R} \theta_yf(\theta_y u,0)\,dy\) \,d\Po(x,f,u) = \\
=\int\(\limsup_{R\to+\infty} \dashint_{B_R} f(u,y)\,dy\) \,d\Po(x,f,u).\end{multline*}
The same multiparameter ergodic theorem allows the more general families of sets described above. The second inequality in \eqref{binf} is trivial.
\end{proof}

We will now prove that under an ergodicity assumption the convergence holds almost surely, not only along subsequences. This is because under our assumptions, the ergodic theorem allows to identify a unique limit.

\begin{theo}\label{th2}  Assume that $\{\f\}$ is as in the previous proposition, and assume in addition that $\f$ $\Gamma$-converges as $\ep\to 0$ to  $\fb$, uniformly with respect to $\om$, $x$. Assume also that  the action $(\om,y)\to\om+y$ is ergodic.

Then, denoting $\Qe$  the marginal  of $\Pe$ with respect to the variables $(x,f)$, it holds almost surely that  
$$ \lim_{\ep\to 0} \Qe = \dashint_G\int\delta_{(x,\fb)}\,d\om\,dx.$$

In particular, if  $\{\Eo(\ue)\}$ is bounded then \eqref{binf} becomes, for almost every $\om_0$, 
\begin{equation}\label{relax}
\liminf_{\ep\to 0} \Eo(\ue) \ge \dashint_G\int\( \inf_{u\in X} \(\limsup_{R\to+\infty} \dashint_{B_R} \fb(u,y)\,dy\)\) \,d\om\,dx.\end{equation}
\end{theo}
The integrand in the right-hand side of \eqref{relax} can be seen as the effective ``infinite cell problem".

\begin{proof} From the uniform convergence of $\f$ to $\fb$, it is not difficult to check that $\lim_{\ep\to 0} \Qe$ exists iff  $ \lim_{\ep\to 0} \Qeb$ exists, and that both limits must then be equal, where
$$\Qeb := \dashint_G \delta_{(x,\fbe)}\,dx.$$

Then fix  a  function $\vp$ which is bounded and Lipschitz-continuous. We first  prove, and this is the essential fact,  that  almost surely, 
\begin{equation}\label{limvp}\lim_{\ep\to 0} \dashint_G \vp(x,\fbe)\,dx = \dashint_G\int \vp(x,\fb)\,d\om\,dx.\end{equation}
 The convergence will follow from the ergodic Theorem of Nguyen and Zessin \cite{nz} (see also the book by U.Krengel \cite{krengel}, Theorem 2.13) which implies that, given $x\in G$, $\eta>0$,  the local averages
$$ \dashint_{B_\eta} \vp(x,\underline f^{\om+\frac{x+h}\ep,x})\,dh$$
converge almost surely. Since the limit must clearly be invariant under $\om\to \om+y$ for every $y$, the hypothesis of ergodicity implies that it is equal almost surely to its expectation, and therefore 
\begin{equation}\label{pointwise}\forall x\in G,\forall\eta>0,\quad \lim_{\ep\to 0} \dashint_{B_\eta} \vp(x,\underline f^{\om+\frac{x+h}\ep,x})\,dh = \int \vp(x,\fb)\,d\om,\quad\text{almost surely}.\end{equation}

To make good use of this fact we transform the left-hand side of \eqref{limvp} by using local averages. We have 
\begin{equation}\label{jeps}I_\ep^\om := \dashint_G \vp(x,\fbe)\,dx  = \dashint_G\dashint_{B_\eta} \vp(x+h,\underline f^{\om+\frac{x+h}\ep,x+h})\,dh\,dx + O(\eta),\end{equation}
the error $O(\eta)$ being due to that part of the integral occuring in an $\eta$-neighbourhood of $\partial G$. In the course of this  proof we will denote $O(a)$ a quantity bounded by $C a$, where $C$ depends only on quantities which are fixed in the proof. In \eqref{jeps}, we have $O(\eta)<C\eta$ where  the constant $C$ depends only on $G$ and $\vp$ which remain fixed throughout the proof.

It is not dificult to check, from the boundedness of $\vp$, that the inner integral 
\begin{equation}\label{aeps}  \dashint_{B_\eta} \vp(x+h,\underline f^{\om+\frac{x+h}\ep,x+h})\,dh\end{equation}
is a Lipschitz function of $x$, with Lipschitz constant bounded by  $C/\eta$. Therefore  the right-hand side integral in \eqref{jeps} may be computed by sampling on a grid: For any $x_0\in\mr^d$ and any for any $\ell>0$ we obtain 
\begin{equation}\label{isum}I_\ep^\om =  \overline\sum_{x\in G\cap x_0+\ell\mz^d} \(\dashint_{B_\eta} \vp(x+h,\underline f^{\om+\frac{x+h}\ep,x+h})\,dh\)+  O\(\eta+\frac\ell\eta\).\end{equation}

The expression in \eqref{aeps}  does not quite agree with the one in \eqref{pointwise} for which we have convergence from the ergodic theorem. To relate them we need to use the uniform measurability hypothesis. For any $\delta>0$ let 
$$A_{\eta,\delta} = \left\{(x,h)\in G\times B_\eta\mid \sup_\om d(\underline f^{\om,x+h},\underline f^{\om,x}) > \delta\right\}.$$
Since the family $\{\f\}$ is uniformly measurable, so is the family $\{\fb\}$, which precisely means that the measure of the slices $\{x\mid (x,h)\in A_{\eta,\delta}\}$ tends to zero as $\eta$ tends to $0$, uniformly with respect to $h$, for any fixed $\delta>0$. By integrating with respect to $h\in B_\eta$ the measure of these slices it follows that for any $\delta>0$, 
\begin{equation}\label{unifmeas}\lim_{\eta\to 0}  \frac{|A_{\eta,\delta}|}{|G||B_\eta|} = 0.\end{equation}
Let us denote by $\mu_{x_0,\ell,\eta}$ the uniform probability measure on $G\cap x_0+\ell\mz^d$ tensored with the uniform probability measure on $B_\eta$. Then 
$$ \frac{|A_{\eta,\delta}|}{|G||B_\eta|} = \dashint_{[0,\ell[^d}\mu_{x,\ell,\eta}(A_{\eta,\delta})\,dx,$$
therefore there exists $\mathbf x = \mathbf x(\ell,\eta,\delta)$ such that 
$$ \mu(A_{\eta,\delta}) \le \frac{|A_{\eta,\delta}|}{|G||B_\eta|} ,\quad \text{where}\quad \mu := \mu_{\mathbf x,\ell,\eta}.$$ 
From the definition of $A_{\eta,\delta}$, if $(x,h)\notin A_{\eta,\delta}$ then $d(\underline f^{\om,x+h},\underline f^{\om,x}) \le\delta$ for any $\om$. Inserting this information in \eqref{isum} we find, letting $G_{\ell,\eta} = G\cap( \mathbf x+\ell\mz^d) $, that 
$$I_\ep^\om =  \overline\sum_{x\in G_{\ell,\eta}} \(\dashint_{B_\eta} \vp(x,\underline f^{\om+\frac{x+h}\ep,x})\,dh\) + O\(\eta+\frac\ell\eta+ \delta + \frac{|A_{\eta,\delta}|}{|G||B_\eta|}\).$$
We may now use \eqref{pointwise} for each of the points in $G_{\ell,\eta}$ to deduce that, for any $\eta,\ell, \delta>0$, almost surely, 
\begin{equation}\label{jsum}I_\ep^\om =  \overline\sum_{x\in G_{\ell,\eta}} \(\int \vp(x,\underline f^{\om,x})\,d\om\) + o_\ep(1)+ O\(\eta+\frac\ell\eta+ \delta + \frac{|A_{\eta,\delta}|}{|G||B_\eta|}\).\end{equation}

On the other hand, we may discretize in the same way the right-hand side of $\eqref{limvp}$: The choice of $\mathbf x$ insures as above that for any $\om$ we have 
\begin{equation*}
\begin{split}
\dashint_G \vp(x,\fb)\,dx  &= \dashint_G\dashint_{B_\eta}  \vp(x+h,\underline f^{\om,x+h})\,dh\,dx + O(\eta)\\
&= \overline\sum_{x\in G_{\ell,\eta}}\dashint_{B_\eta}  \vp(x+h,\underline f^{\om,x+h})\,dh + O\(\eta+\frac\ell\eta\)\\
&= \overline\sum_{x\in G_{\ell,\eta}}\dashint_{B_\eta} \vp(x,\underline f^{\om,x})\,dh  + O\(\eta+\frac\ell\eta+\delta + \frac{|A_{\eta,\delta}|}{|G||B_\eta|}\).
\end{split}
\end{equation*}
Since the integrand on the last line is independent of $h$, the integral with respect to $h$ may be removed, and integrating with respect to $\om$ we find that 
$$\dashint_G\int \vp(x,\fb)\,d\om\,dx = \overline\sum_{x\in G_{\ell,\eta}} \int \vp(x,\underline f^{\om,x})\,d\om + O\(\eta+\frac\ell\eta+\delta + \frac{|A_{\eta,\delta}|}{|G||B_\eta|}\).$$
This together with \eqref{jsum} proves that for any for any $\eta,\ell, \delta>0$, almost surely, 
$$\limsup_{\ep\to 0} \left|I_\ep^\om - \dashint_G\int \vp(x,\fb)\,d\om\,dx\right|\le \(\eta+\frac\ell\eta+\delta + \frac{|A_{\eta,\delta}|}{|G||B_\eta|}\).$$
Therefore, almost surely, this inequality will hold true  for $\eta,\ell,\delta$ belonging to the countable set $\{1/n,\ n\in\mn\}$. Letting $\ell$, then $\eta$, then $\delta$ tend to $0$ along this sequence, we deduce using in particular \eqref{unifmeas} that \eqref{limvp} holds.

To conclude the proof of the theorem, we choose a countable dense set of test-functions $\vp$, and note that from the above, almost-surely, \eqref{limvp} holds for every $\vp$ in this set. Now choose $\om$ such that this is the case. Then from Propostion~\ref{th1} and for any sequence $\{\ep_n\}_n$ tending to $0$, a subsequence $\{Q_{\ep_{n'}}^\om\}_{n'}$ converges as $n'\to +\infty$ to some probability measure $Q^\om$. But because \eqref{limvp} holds for a countable dense set of test-functions, we have 
$$Q^\om = \dashint_G\int\delta_{(x,\fb)}\,d\om\,dx.$$
In particular $Q^\om$ is clearly independent of the subsequence. Therefore the whole family $\{\Qe\}_\ep$ converges to $Q^\om$, and this holds almost-surely. 

The lower-bound \eqref{relax} is then simply a restatement of \eqref{binf}.
\end{proof}

\section{Convex random homogenization} \label{convex} In this section we establish all the hypotheses required by our framework for functionals of the type \eqref{dalmasosetting},  and how one can then recover the lower bound in  \cite{dalmasopaper,dmm}. In this case, the family $\{\f\}$ turns out not to depend on $x$ or $\ep$.

For the convenience of the reader we recall the lower-bound part of Theorem~I  in \cite{dalmasopaper}, slightly modifying the language used and the fact that we  replace the action of $\mathbb Z^n$ on $\Om$ by the action of $\mr^n$, which is unimportant as specialists know. 

Note that we have also replaced the assumption that $\fdm( ., y+z, p)$ and $\fdm( ., y, p)$ have same law by the stronger assumption that there exists a group of measure preserving transformations $\om\to\om +y$ such that $\fdm( \om, y+z, p)$ and $\fdm( \om+y, z, p)$.

In the following let 
\begin{equation}\label{QR} Q_R = (-R,R)^n.\end{equation}

\begin{theo}[\cite{dalmasopaper}]\label{dalmasoth} Let $F_\ep^\om$ be a random integral functional defined by 
\begin{equation}\label{lagdm}F_\ep^\om(v)= \dashint_G \fdm( \om, x/\ep, \nabla v(x)) dx,\end{equation}
where the Lagrangian $\fdm(\om, y,q)$ is a positive function convex in $q$ and measurable in $y$,   satisfying the stationarity and growth conditions
\begin{equation}\label{growth}\fdm(\om, z+y ,p) = \fdm(\om+y,z,p),\quad c_0 |q|^p \le \fdm(y,q) \le  C_0 (1+|q|)^p ,\end{equation}
and the action of $\mathbb{R}^n$ on $\Omega$ is ergodic.

Then  for almost every $\om$ and any sequence $\{v_\ep\}_\ep$ such that $\{F_\ep^\om(v_\ep)\}_\ep$ is bounded, we have \begin{equation} \label{lbdm}\liminf_{\ep\to 0} F_\ep^\om(v_\ep) \ge  \dashint_G f_*(\nabla v(x)) dx.\end{equation}
Here $f_*(q)$ satisfies the growth conditions \eqref{growth}  and is computed as follows
\begin{equation}\label{foq}f_*(q) := \lim_{R\to+\infty} \min_{\substack{u:Q_R\to\mr\\ \text{$u(y) =q\cdot y$ on $\partial Q_R$}}} \dashint_{Q_R} \fdm( \om, y, \nabla u(y)) dy.\end{equation}
Due to the ergodicity, the right-hand side of \eqref{foq} is constant a.e. with respect to  $\om.$
\end{theo}

This theorem shows that for Lagrangians of the form \eqref{lagdm}, the general bound \eqref{step1}  implies the simpler  lower bound \eqref{lbdm}  which can be computed via expanding cell problems.

\subsection{Extended coercive integral functionals}

To recast the above problem in our setting, we follow the procedure sketched in Section~\ref{recast}.

For some given $p>1$, we let $X$ be the space of functions on $\mr^n$ modulo constants (i.e. two functions which differ by a constant are considered equal) on which the topology is that of $\Lp$ convergence. Thus a sequence $\{u_n\}_n$ converges to $u$ if there exists a sequence of real numbers $\{c_n\}_n$ such that $u_n +c_n\to u$ in $L^p(K)$ for any compact subset $K$  of $\mr^n$. The space  $\Lp$ is a separable metric space for the distance 
\begin{equation}\label{dp} d_p(f,g) = \sum_{n=1}^{+\infty} 2^{-n} \min\(\|f-g\|_{L^p(B(0,n))},1\).\end{equation}
From this distance we deduce a distance on $X$ defined as 
$$ d_X(f,g) = \inf_{c\in\mr} d_p(f+c, g) =  \inf_{c,c'\in\mr} d_p(f+c, g+c').$$
The last equality insures that $d_X$ is symmetric and satisfies the triangle inequality. The space $X$ is obviously complete and separable.

On $X$ we consider the class $\F_0$ of functionals of the type 
\begin{equation}\label{fo} f(u,y) = \begin{cases}\D \dashint_{B(y,1)} \fdm(z, \nabla u(z))\,dz &\text{if $\nabla u\in L^p(B(y,1))$} \\ +\infty &\text{otherwise,}\end{cases}\end{equation}
where $\fdm(y,q)$ is a positive function convex in $q$ and measurable in $y$ satisfying the growth condition 
$$c_0 |q|^p \le \fdm(y,q) \le  C_0 (1+|q|)^p .$$
Here $c_0$ and $C_0$ are fixed positive constants. 

The action $\theta$ on $X$ is $\theta_y u =u(\cdot+y)$, from which we deduce through \eqref{act} that 
\begin{equation}\label{trans} \theta_{y'} f(u,y) = \begin{cases}\D \dashint_{B(y,1)} \fdm(z+y',\nabla u(z))\,dz &\text{ if $\nabla u\in L^p(B(y,1))$ }\\ +\infty &\text{ otherwise.}\end{cases}
\end{equation}

We check the hypotheses necessary to apply our framework.
\begin{lem} The action $\theta$ is continuous with respect to $y$ and uniformly continuous with respect to $u$ relatively to $y\in K$, for any bounded $K\subset\mr^n$.\end{lem}
\begin{proof} If $y_n$ converges to $y$, $u\in X$ and $R>1$  then $u\in L^p(B(y,R))$ thus $\theta_{y_n}u\to\theta_{y} u$ in $L^p(B(y,R-1))$. Since this is true for any $R$, we have $\theta_{y_n} u\to \theta_y  u$ in $\Lp$.

We now prove the uniform continuity in $u$. Assume $R>0$ and let $\{y_n\}$ be a sequence in $B(0,R)$. If $u_n\to u$ in $X$, there exists constants $\{c_n\}_n$ such that we have for any $R'>0$ that $ u_n+c_n\to u$ in $L^p(B(0,R'))$ therefore
$\theta_{y_n} (u_n +c_n) - \theta_{y_n} u$ tends to $0$ in $L^p(B(0,R'))$. Since this is true for any $R'$, we obtain the convergence of $\theta_{y_n} (u_n+c_n) - \theta_{y_n} \nabla u$ to $0$ in $\Lp$, thus the convergence of $u_n$ to $u$ in $X$  and the desired uniform continuity.
\end{proof}

We also have 
\begin{lem} The functionals in $\F_0$ are lower semicontinuous, and bounded below by a lower semicontinuous coercive functional $f_0$.\end{lem}
\begin{proof} We begin with the lower semicontinuity. Assume that $f\in\F_0$, that $u_n\to u$ in $X$ (i.e. $u_n+c_n\to u$ in $\Lp$) and that $y_n\to y$. If $\liminf_n f(u_n,y_n) = +\infty$ there is nothing to prove. Otherwise, we consider a subsequence (not relabeled) which realizes the $\liminf$, hence satisfies $f(u_n,y_n)\to \ell\in\mr_+$. Then any ball $B$ such that $\overline B\subset B(y,1)$ is included in $B(y_n,1)$ if $n$ is large enough hence, letting $\fdm$ be the Lagrangian associated to $f$, we have 
$$\limsup_n \dashint_B \fdm(z,\nabla u_n)\,dz \le \limsup_n \dashint_{B(y_n,1)} \fdm(z, \nabla u_n)\,dz = \limsup_n f(u_n,y_n) = \ell.$$
It follows that $\{u_n+c_n\}$ is bounded in $W^{1,p}(B)$ for any such $B$ and then that a subsequence converges weakly in $W^{1,p}(B)$ and strongly in $L^p(B)$ by compact embedding, to $u$. Moreover,
$$\dashint_B \fdm(z,\nabla u)\,dz \le \liminf_n \dashint_{B} \fdm(z, \nabla u_n)\,dz \le \ell.$$
It follows by taking a sequence $B_k\nearrow B(y,1)$ that $u\in W^{1,p}(B(y,1))$ and that $f(u,y)\le \ell$.

As a coercive lower semicontinuous functional bounding from below every $f\in\F_0$ we choose 
$$ f_0(u,y) = \begin{cases} +\infty & \text{If $u\notin W^{1,p}(B(y,1))$}\\\displaystyle
c_0\dashint_{B(y,1)}  |\nabla u(z)|^p \,dz & \text{otherwise}\end{cases}$$
It is clear that $f_0$ bounds $f$ from below. The lower semicontinuity of $f_0$ is proved as above. To see that $f_0$ is coercive, we assume that $\{C_R\}_{R>0}$ are arbitrary positive numbers and that $\{u_n\}_n$ is such that $\dashint_{B_R} f_0(u_n) \le C_R$ for every $R$. 

Then   for any $R>0$, $\{\nabla u_n\}_n$ is bounded in $L^p(B_R)$, and we may choose $c_n\in\mr$ such that for instance $\dashint_{B(0,1)} (u_n+c_n) = 0$. This together with the gradient bound implies for any fixed $R$ a bound in $L^p(B_R)$ for $u_n+c_n$, by a generalized Poincaré inequality as \cite{GT}~(7.45). Thus $\{u_n+c_n\}$ is bounded in $W^{1,p}(B_R)$   and there exists  a subsequence which converges strongly in $L^p(B_R)$. Then, also, by a diagonal argument, the existence of a subsequence which converges in $\Lp$ to some $u$ follows (note that $c_n$ does not depend on $R$) . It follows that $u_n\to u$ in $X$, and thus  the set of $u$'s satisfying $\dashint_{B_R} f_0(u) \le C_R$ for every $R$ is relatively compact, and  $f_0$ is coercive.
\end{proof}

We pursue with  the less trivial  (M) property.
\begin{pro} The family $\F_0$ satisfies the (M) property.\end{pro}
\begin{proof} Assume $\lambda$, $M$ are positive and that $(u_0,y_0)\in \Lp\times\mr$. For any $(u,y)\in\Lp\times \mr$ we let $c\in\mr$  minimize $\|u+c-u_0\|_{L^p(B(y,1))}$  and 
$$v =  \begin{cases} u+c  & \text{on  $B(y,1)$, }\\
 u_0 & \text{elsewhere.}\end{cases}$$
 Then $f(u,y) = f(v,y)$ and $d_X(u,u_0)\ge d_X(v,u_0)$ hence
 $$F^\lambda(u,y) := f(u,y) + \lambda \(d_X(u,u_0) + |y-y_0|\)  \ge f(v,y) + \lambda \(d_X(v,u_0) + |y-y_0|\).$$
 This implies that 
 $$R_\lambda f(u_0,y_0) = \inf_{\Lp\times\mr} F^\lambda = \inf_{(u,y)\in\Lp\times\mr} F^\lambda(v,y).$$
 Therefore, assuming $R^\lambda f\le M$, we find  that the last infimum may be taken over the set $K$ of $v$'s and $y$'s such that 
 \begin{equation}\label{optimalc}\text{or any $c\in\mr$,}\quad \|v-u_0\|_{L^p(B(y,1))}\le \|v+c-u_0\|_{L^p(B(y,1))},\end{equation}
 and 
 $$|y-y_0|\le \frac{2M}\lambda,\quad f_0(v,y)\le M,\quad \text{$v=u_0$ outside $B(y,1)$}.$$
 Let us now show that this set is compact. Assume $\{y_n\}$ and $\{v_n\}$ satisfy the above, then after extracting a subsequence we first  have $y_n\to y$.
 
 Then the bound $f_0(v_n,y_n)\le M$ implies that the norm of $\nabla v_n$ in $L^p(B(y_n,1))$ is bounded. In turn this implies that, denoting by $\bar v_n$ the average of $v_n$ over $B(y_n,1)$, the $L^p$ norm of $v_n - \bar v_n$ over $B(y_n,1)$ is bounded. Using \eqref{optimalc} with $c=-\bar v_n$ we find that $v_n$ is bounded in $L^p(B(y_n,1))$ too, and then that $v_n$ is bounded in $W^{1,p}(B(y_n,1))$.
 
 It follows that  $\theta_{y-y_n} v_n$ is bounded in $W^{1,p}(B(y,1))$ and thus converges in $L^p(B(y,1))$ after extraction. But  $\theta_{y-y_n} v_n= \theta_{y-y_n} u_0$ outside $B(y_n,1)$, thus we have  convergence of $\theta_{y-y_n} v_n$ to some $v$ in $\Lp$. Since $y_n-y \to 0$  we deduce that $v_n\to v$ in $\Lp$. From the lower semicontinuity of $f_0$ it is clear that $(v,y)\in K$, which is therefore compact. This proves property (M) since $K$ is independent of $f$. 
 \end{proof}
 
 It is well known since the paper  \cite{bd}  that the $\Gamma$-limit of an integral functional as defined above is an integral functional of the same type (with the same constants $c_0$ and $C_0$). Thus  $\F_0$ is closed under $\Gamma$-convergence and the above shows that the distance $d$ both metrizes $\Gamma$-convergence on $\F_0$ and makes it a  compact metric space. This is really a restatement of \cite{dmm}.
 
\subsection{Lower bound before convexity}
We now consider a random functional, i.e. a family  $\F = \{f^\om\}_\om$ of functionals in $\F_0$ parametrized by the random parameter $\om$, such that $\om\mapsto f^\om$ is measurable. We have 
\begin{equation}\label{fom} f^\om(u,y) = \begin{cases}\D \dashint_{B(y,1)} \fdm(\om, z, \nabla u(z))\,dz &\text{if $\nabla u\in L^p(B(y,1))$} \\ +\infty &\text{otherwise.}\end{cases}\end{equation}
We assume that the family is stationary i.e. that $\theta_y f^\om = f^{\om+y}$ or, equivalently in view of \eqref{trans}, that 
$$\fdm(\om, z+y ,p) = \fdm(\om+y,z,p).$$

Since we have no dependence on the slow parameter $x$, we need not check the uniform measurability with respect to $x$, and the results of the previous section apply. However, because   our general setting applies to non-convex functionals as well as convex ones, we cannot expect to derive   the result of  \cite{dalmasopaper} directly  from it (even excluding the upper bound part). What we get is rather  an intermediary result which, processed by using the convexity hypothesis, will yield the lower bound part of Theorem~\ref{dalmasoth}.

We let, as in Theorem~\ref{dalmasoth}, 
$$F_\ep^\om(v)= \dashint_G \fdm(\om, x/\ep, \nabla v(x)) dx,$$
and we  assume that $\{v_\ep\}$ is a sequence in $W^{1,p}(G)$ such that $F_\ep^\om(v_\ep, G)\le C$, with $C$ independent of $\ep$. Note that this bound is independent of $\om$ because of the growth  assumption on $\fdm$. We then let
\begin{equation}\label{rescale}
 u_\ep(y) = \ep^{-1} v_\ep(\ep y).
 \end{equation}
We may extend $v_\ep$ to  $G_\ep = \{x\mid d(x,G) <\ep\}$ in such a way that, as $\ep\to 0$, 
$$
F_\ep^\om(v_\ep) \ge  \dashint_{G} f^\om(u_\ep,\frac{x}{\ep}) \,dx - o(1), 
$$
where $f^\om$ is defined in \eqref{fom}.

We now deduce from Proposition~\ref{th1} that 

\begin{pro}\label{beforeconvex} Under the hypotheses of Theorem~\ref{dalmasoth}, and using the notation there, if  $\{v_\ep\}$ is a sequence in $W^{1,p}(G)$ such that $F_\ep^\om(v_\ep, G)\le C$ and if $\{\ep_n\}_n$ is a sequence tending to $0$ such that $\Pen\to \Po$ as $n\to +\infty$, then 
\begin{equation}\label{step1}
\liminf_{n\to+\infty} F_{\ep_n}^\om(v_{\ep_n}) \ge \int\(\limsup_{R\to+\infty} \dashint_{Q_R} f(u,y)\,dy\) \,d\Po(x,f,u),
\end{equation}
where $Q_R = (-R,R)^n$ and 
\begin{equation}\label{po}\Po = \lim_{n \to +\infty} \Pen,\quad \Pe = \dashint_G \delta_{\(x,\theta_\xe\f,\theta_\xe\ue\)},\end{equation}
with $\ue$ defined by \eqref{rescale}.
\end{pro}
It remains to use the particular structure \eqref{lagdm} of our problem to see how this yields  the lower bound in the $\Gamma$-convergence stated in  Theorem~\ref{dalmasoth}.

\subsection{Processing the lower bound}
In this section we assume the hypotheses of Theorem~\ref{dalmasoth} and let, as in Proposition~\ref{beforeconvex}, 
$$\Po = \lim_{n \to +\infty} \Pen,\quad \Pe = \dashint_G \delta_{\(x,\theta_\xe\f,\theta_\xe\ue\)}.$$

We introduce the following notation: for any function $u:\mr^n\to\mr$ and any $R>0$ we let 
\begin{equation}\label{ur} u_R(x) = \frac1R u(R x), \quad u_{R,y}(x)= \frac1Ru(y+Rx).\end{equation}
We will use this notation either to let $R$ go to $+\infty$ (blow down), or $R\to 0$ (blow-up, in which case we will use lower case $r$ instead of $R$).

We recall the following $L^{p^*}$ differentiability property  of functions in  $W^{1,p}_\loc$ (see \cite{ziemer} or \cite{eg}):  
\begin{pro} Any $u\in W^{1,p}_\loc$ is such that for almost every $y\in\mr^n$ and as $r\to 0$, 
\begin{equation}\label{difflp} \psi(u,y,r):= \dashint_{B(0,1)}\left|u_{r,y}(x) - \nab u(y)\cdot x\right|^{p^*}\quad\longrightarrow\quad 0,\end{equation}
where $p^* = np/(n-p).$
\end{pro}

A consequence of this is the following
\begin{pro} \label{limblowdown} Assume $P$ is a translation invariant probability measure on the space $X$ of $\Lp$ functions modulo constants such that $P$-almost every $u$ is in $W^{1,p}_\loc$ and 
\begin{equation}\label{sobolev}\int\(\dashint_{B(0,1)} |\nabla u|^p\)\,dP(u) < +\infty.\end{equation}
Then  there exists $R_n\to +\infty$ such that  $P$-almost every $u$ is such that $d(u_{R_n})\to 0$ as $n\to +\infty$, where $d(u)$ denotes the distance in $\Lp$ of $u$ to the set of linear maps.
\end{pro}
\begin{proof}Let $\{R_n\}$ be any sequence tending to $+\infty$, and let $P_{R_n}$ be the push-forward of $P$ by the map $u \mapsto u_{R_n}$. 
We may check that  $\{P_{R_n}\}_n$ is tight, the proof mimics that of  Proposition~\ref{tight} and uses \eqref{sobolev},   we omit it. The tightness implies that any subsequence has a convergent subsequence. Denote by $P_\infty$ the weak limit of a sequence $\{P_{R_n}\}_n$, such that $R_n\to +\infty$. 

First,  $P_\infty$ is translation-invariant, indeed for any bounded continuous function $\vp$ on $X$ and any $x\in\mr^n$  we have 
\begin{multline*}\int \vp( \theta_x u)\,dP_\infty( u) = 
\lim_R \int \vp( \theta_x u)\,dP_R(u) = \\
\lim_R \int \vp( \theta_x u_R)\,dP(u)=
\lim_R \int \vp\((\theta_{Rx}u)_{R}\)\,dP(u) .
\end{multline*}
using the fact that $P$ is translation-invariant, this is equal to 
$$\lim_R \int \vp( u_R)\,dP(u) = \int \vp(u)\,dP_\infty(u),$$
thus $P_\infty$ is translation-invariant.

Second, because of the $L^{p^*}$ differentiability property of functions in $W^{1,p}$, we have for $P_\infty$-a.e. $u$ that  for a.e. $y\in\mr^n$ that $\lim_{r \to 0} \psi(u,y, r) =0$, where $\psi$ is defined in \eqref{difflp}. Then, by Lebesgue's dominated convergence theorem, 
$$\lim_{r\to 0} \iint \min(\psi(u,y,r), 1) dP_\infty(u)\,dy  =\lim_{r\to 0} \iint \min(\psi(u,y,r), 1) dy\, dP_\infty(u) =0.$$
Using the translation-invariance of $P_\infty$, the inner integral on the left-hand side is independent of $y$ hence we deduce  
\begin{equation}\label{limpsi}\lim_{r\to 0} \int \min(\psi(u,0,r), 1) dP_\infty(u) = 0.\end{equation}

Now  for any integer $k$ and using Hölder's inequality we have 
$$\|u_r - \nab u(0)\|_{L^p(B(0,k)} \le |B_1|^{\frac1p} k^{\frac np - 1}\psi(u,0,kr)^{\frac1{p^*}}.$$
Therefore, in view of \eqref{dp}, 
$$\int d_p(u_r,\nab u(0))\,dP_\infty(u)  \le C\sum_k \(2^{-k}k^{\frac np - 1}\int \min\(\psi(u,0,kr)^{\frac1{p^*}},1\)\,dP_\infty(u)\).$$
Thus, in view of \eqref{limpsi}, we find that 
$$\lim_{r\to 0} \int d(u_r)\, dP_\infty(u) =0,$$
where $d(u)$ denotes the distance in $\Lp$ of $u$ to linear maps.

Finally, from the definition of $P_\infty$, it follows that 
$$\lim_{r\to 0} \lim_{n\to +\infty} \int d(u_{r R_n}) \, dP(u) =\lim_{r\to 0} \lim_{n\to +\infty} \int d(u_r) \, dP_{R_n}(u) =0,$$
and  by choosing $r_k = 1/k$ and for each  integer $k$ by choosing $n_k$ large enough, we deduce the existence of a sequence $\{R'_k\} = \{r_k R_{n_k}\}$ which tends to $+\infty$ and such that 
$$\lim_{k\to +\infty} \int d(u_{R'_k}) \, dP(u) =0.$$
Going to a subsequence, we find that for $P$-a.e. $u$ we have $\lim_{k\to +\infty} d(u_{R'_k}) = 0$. 
\end{proof}

Applying Proposition~\ref{limblowdown}  to $d\Po(x,f,u)$, or rather to its marginal $dQ^\om(u)$ with respect to $u$, we deduce that there exists a sequence $\{R\}$ tending to $+\infty$ such that for $\Po$-almost every $u$, the blow-down maps $u_R$ get $\Lp$ closer and closer to a linear map as $R\to+\infty$.

The rest of this section is devoted to the proof of Theorem~\ref{dalmasoth}.

\subsubsection*{Step 1: asymptotic linearity} We start by showing that $\Po$-almost surely, the profiles $u$ are asymptotically linear in a suitable sense. 

Since
$$\int f(u,y)\,d\Po(x,f,u)<+\infty,$$
and since 
$$f(u,y)\ge c_0\dashint_{B(y,1)}|\nab u(z)|^p\,dz,$$
we deduce that 
\begin{equation}\label{borne}\int\(\dashint_{B(0,1)} |\nab u(z)|^p\,dz\)\,d\Po(x,f,u) < +\infty.\end{equation}
It then follows from Proposition~\ref{limblowdown}  that there then exists a sequence $\{R\}$ tending to $+\infty$ such that the distance in $\Lp$ of $u_R$ to the set of linear maps tends to $0$ as $R\to +\infty$ along the sequence. Thus for $P$-almost every $u$ and every $R$ in the sequence there exists $q_{u,R}\in\mr^n$ such that $\|u_R - q_{u,R}\cdot x\|_{L^p(Q_1)}$ tends to $0$ as $R\to +\infty$, where we identified the vector $q_{u,R}$ with the linear map $x\mapsto q_{u,R}\cdot x$.

On the other hand the ergodic theorem implies that for $\Po$-almost every $u$, 
\begin{equation}\label{avgrad}\text{$\displaystyle \lim_{R\to+\infty} \dashint_{Q_R} \nab u(y) \,dy$ exists and is finite.}\end{equation}
If we denote $q(u)$ this limit, then we must have $q(u) = \lim_{R\to +\infty} q_{u,R}$ along the sequence $\{R\}$ above.

It follows that   there  exists a sequence $\{R\}$ such that for almost every $u$, 
\begin{equation}\label{asymplin}\text{$\|u_R - q(u)\|_{L^p(Q_1)}$ tends to $0$ as $R\to +\infty$},\quad q(u) = \lim_{R\to +\infty}\dashint_{Q_R} \nab u.\end{equation}
Note that the ergodic theorem also implies that for a.e. $u$ we have 
\begin{equation}\label{limerg}\lim_{R\to +\infty}\dashint_{Q_R} |\nab u|^p,\quad\text{exists}.\end{equation}

\subsubsection*{Step 2: Relaxation of \eqref{step1}} Let us define for any $q\in\mr^n$ and any Lagrangian $\gdm(y,q)$ which is a positive function, convex in $q$ and measurable in $y$ such that $ c_0 |q|^p \le \gdm(y,q) \le  C_0 (1+|q|)^p$, the quantity
\begin{equation}\label{mruf}m_R(q,\gdm) := \min_{\substack{v:Q_R\to\mr\\ \text{$v(y) = q\cdot y$ on $\partial Q_R$}}} \dashint_{Q_R} \gdm(y, \nabla v(y)) dy.\end{equation}
We know that $\Po$-almost every functional $g$ is of the form \eqref{fo} for some Lagrangian $\gdm$. We will now abuse notation by identifying $g$ with $\gdm$. With this identification, \eqref{step1}  may be rewritten
$$\liminf_{n\to +\infty} F_{\ep_n}^\om(v_{\ep_n}) \ge \int\(\limsup_{R\to+\infty} \dashint_{Q_R} \dashint _{B(0,1)} \gdm(y,\nab u(y))\,dy\) \,d\Po(x,\gdm,u),$$
which easily yields (we omit details)
\begin{equation}\label{step2}
\liminf_{n\to +\infty} F_{\ep_n}^\om(v_{\ep_n}) \ge \int\(\limsup_{R\to+\infty} \dashint_{Q_R} \gdm(y,\nab u(y))\,dy\) \,d\Po(x,\gdm,u).
\end{equation}
The goal of this step is to prove that  this inequality implies
\begin{equation}\label{step3}
\liminf_{n\to +\infty} F_{\ep_n}^\om(v_{\ep_n})\ge \int\(\limsup_{R\to+\infty} m_R(q(u), \gdm)\) \,d\Po(x,\gdm,u), 
\end{equation}
where $q(u)$ is defined in \eqref{asymplin}.

To prove this, let $u$ be such that  \eqref{asymplin}, \eqref{limerg} hold for some sequence $\{R\}$ tending to $+\infty$, which is true  for a.e. $u$. The limits $R\to +\infty$ below will always be taken along this sequence. Choose  $\delta\in (0,1)$ and let $R>1$. Let us define a function $v_\delta$ on $Q_R$ such that $v_\delta(y) = q(u) \cdot y$ on $\partial Q_R$ and which is sufficiently close to $u$. Then $v_\delta$ can be used as a test function in the definition of $m_R(q(u), \gdm)$ to bound  $m_R(q(u), \gdm)$ from above in terms of the integral over $Q_R$ of $\gdm(y,\nab u(y))$. The function $v_\delta$ is defined to be such that
\begin{equation}\label{vdelta} v_\delta(y) = q(u)\cdot y + \chi_\delta(\frac yR) (u(y) - q(u)\cdot y),\end{equation}
where $\chi_\delta$ is a cutoff function independent of $R$ defined on $Q_1$, takes values between $0$ and $1$, is equal to $0$ on the boundary, and equal to $1$ on $Q_{1-\delta}$. We also choose $\chi_\delta$ so that $|\nab v_\delta|\le 2/\delta$. 

We then compute 
$$\int_{Q_R} \gdm(y,\nab v_\delta(y))\,dy \le \int_{Q_{R(1-\delta)}} \gdm(y,\nab u(y))\,dy + \int_{Q_R\sm Q_{R(1-\delta)}} \gdm(y,\nab v_\delta(y))\,dy.$$
From the growth condition on $\gdm$ and computing $\nab v_\delta$ we deduce using standard arguments that 
\begin{equation}\label{euv}\int_{Q_R} \gdm(y,\nab v_\delta(y))\,dy \le \int_{Q_R} \gdm(y,\nab u(y))\,dy +\text{remainder},\end{equation}
where
\begin{equation}\label{remainder}\text{remainder} \le C \(\delta |Q_R|+ \delta {q(u)}^p |Q_R|+ R^{-p}\|u-q(u)\|_{L^p(Q_R)}^p + \int_{Q_R\sm Q_{R(1-\delta)}} |\nab u|^p\).\end{equation}
From the ergodic theorem, the average of $|\nab u|^p$ on $Q_R$ has a limit as $R\to +\infty$ for  it follows that as $R\to +\infty$
$$\int_{Q_R\sm Q_{R(1-\delta)}} |\nab u|^p \quad \sim\quad \delta \int_{Q_R} |\nab u|^p.$$
Moreover, using \eqref{asymplin}, as $R\to +\infty$
$$R^{-p}\|u(y)-q(u)\cdot y\|_{L^p(Q_R)}^p = R^{n}\|u_R(z)-q(u)\cdot z\|_{L^p(Q_1)}^p = o(R^n).$$
Plugging this information into \eqref{remainder}, \eqref{euv}, dividing by $|Q_R|$ and letting $R\to +\infty$ we find 
\begin{equation}\label{step4}\liminf_{R\to +\infty}m_R(q(u),\gdm) \le \limsup_{R\to +\infty} \dashint_{Q_R} \gdm(y,\nab u(y))\,dy + C\delta \(1+\lim_{R\to+\infty} \dashint_{Q_R} |\nab u(y)|^p\,dy\),\end{equation}
where we have taken into account the fact that $v_\delta$ is a legitimate test function in the definition of $m_R(q(u),\gdm)$ and where the limits no longer need to be along the sequence $\{R\}$ because of the  $\limsup$ and $\liminf$.

It remains to integrate \eqref{step4} with respect to $\Po$ and to let $\delta\to 0$ to find, in view of \eqref{borne} that 
\begin{equation}\label{step5}\liminf_{n\to +\infty} F_{\ep_n}^\om(v_{\ep_n}) \ge \int\(\liminf_{R\to+\infty} m_R(q(u), \gdm)\) \,d\Po(x,\gdm,u).\end{equation}

\subsubsection*{Step 3: Separation of variables}  Now we show why in \eqref{step5}, integration with respect to the variables $(x,u)$ and integration with respect to $\gdm$ separate. Because of the invariance of $\Po$ as stated in Proposition~\ref{inv1}, using the ergodic theorem we may replace the integrand 
$I(q(u),\gdm):=\liminf_{R\to+\infty} m_R(q(u), \gdm)$ in \eqref{step5} by 
$$\tilde I(q(u),\gdm) := \lim_{R\to +\infty} \dashint_{B_R} I(q(\theta_y u), \theta_y \gdm)\,dy.$$
But it is clear that $q(\theta_y u) = q(u)$ for any $y$, therefore, 
$$\tilde I(q(u),\gdm) = \lim_{R\to +\infty} \dashint_{B_R} I(q(u), \theta_y \gdm)\,dy,$$
and it follows from the ergodicity of the action that for $P^\omega$-almost every $(u,\gdm)$ it holds that 
$$\tilde I(q(u),\gdm) = \int I(q(u), \overline \gdm)\,d\Qo(\overline \gdm),$$
where $\Qo$ is the marginal of $\Po$ with respect to the variable $\gdm$.  Thus $\tilde I(q(u), \gdm)$ is independent of $\gdm$. Inserting this relation into
 \eqref{step5}, we deduce that, denoting $\Po(x,u)$ the marginal of $\Po$ with respect to the variables $(x,u)$, 
\begin{equation}\label{step6}\liminf_{n\to +\infty} F_{\ep_n}^\om(v_{\ep_n}) \ge \int\(\int\(\liminf_{R\to+\infty} m_R(q(u), \gdm)\)\,d\Qo(\gdm)\) \,d\Po(x,u).\end{equation}
Applying the sub-additive ergodic theorem as in \cite{dmm}, $\Qo$-almost every $\gdm$ is such that $m_R(q(u), \gdm)$ has a limit as $R\to +\infty$, therefore  we may replace the $\liminf$ in \eqref{step6} by a $\lim$, yielding
\begin{equation}\label{step7}\liminf_{n\to +\infty} F_{\ep_n}^\om(v_{\ep_n})\ge \int\(\int\(\lim_{R\to+\infty} m_R(q(u), \gdm)\)\,d\Qo(\gdm)\) \,d\Po(x,u).\end{equation}

\subsubsection*{Step 4: Convexity} It is now time to make use of the convexity assumption. First we use Fubini's Theorem to write
\begin{multline}\label{fubcon}\iint\(\lim_{R\to+\infty} m_R(q(u), \gdm)\)\,d\Qo(\gdm) \,d\Po(x,u) \\
= \iint\(\lim_{R\to+\infty} m_R(q(u), \gdm)\)\,d\Po(x,u)\,d\Qo(\gdm).\end{multline}
Denote by $\{\Pox\}_x$ the disintegration of $\Po$ with respect to the variable $x$ (see \cite{jiri}, \cite{schwartz} or \cite{ags}) so that $d\Po(x,u) = d\Pox(u)\,d\lambda(x)$ where $\lambda$ is the normalized Lebesgue measure on $G$. Then,   because $\gdm$ is convex, the map $q\to \lim_{R\to+\infty} m_R(q(u), \gdm)$ is convex as well hence we have 
\begin{equation}\label{qx}\int\(\lim_{R\to+\infty} m_R(q(u), \gdm)\)\,d\Pox(u)\ge \lim_{R\to+\infty} m_R(q_x, \gdm),\quad\text{where}\quad q_x = \int q(u)\,d\Pox(u).\end{equation}
Replacing in \eqref{fubcon}, \eqref{step7} we find after applying Fubini's Theorem
$$\liminf_{n\to +\infty} F_{\ep_n}^\om(v_{\ep_n})\ge \dashint_G \(\int\lim_{R\to+\infty} m_R(q_x, \gdm)\,d\Qo(\gdm)\)\,dx,$$
which proves \eqref{lbdm} and Theorem~\ref{dalmasoth} provided we show that for a.e. $x$ we have $q_x = \nab v(x)$.

\subsubsection*{Step 5: Conclusion} To prove that $q_x = \nab v(x)$  a.e. we go back to the definition of $q_x$ in \eqref{qx} and note that since $\Pox$ is translation-invariant and in view of the definition of $q(u)$ in \eqref{asymplin}, the ergodic theorem implies that 
$$ q_x = \int \dashint_{Q_1}\nab u(y)\,dy\,d\Pox(u).$$
Then, for any smooth vector field$\vp(x)$ we have 
$$\dashint_G \vp(x)\cdot q_x\,dx = \dashint_G\int\dashint_{Q_1} \vp(x)\cdot\nab u(y)\,dy\,d\Pox(u)\,dx = \int\dashint_{Q_1} \vp(x)\cdot\nab  u(y)\,dy\,d\Po(x,u).$$
Recalling \eqref{po}, we deduce that
$$\dashint_G \vp(x)\cdot q_x\,dx = \lim_{\ep\to 0}\dashint_G\dashint_{Q_\ep(x)} \vp(x)\cdot\nab v_\ep(y)\,dy\,dx.$$
Indeed,  $P_{\ep_n}$ converges weakly to $P$ and both are supported on a bounded subset of  $W^{1,p}_\loc$ (up to a set of arbitrarily small measure). On the other hand, the restriction  of  $f:u\to\dashint_{Q_1}\nab u$ to such a bounded set is bounded and continuous on $\Lp$, since a sequence $u_n$ of that set that converges in $\Lp$ also converges weakly in $W^{1,p}_\loc$.  We then deduce that $\int f \, dP_{\ep_n} $ converges to $\int f\, dP$ hence the result. 

Passing to the limit in the right-hand side we find 
$$\dashint_G \vp(x)\cdot q_x\,dx = \dashint_G\vp(x)\cdot\nab v(x)\,dx,$$
where $v$ is the weak limit of $\{v_\ep\}_\ep$. Since this is true for any smooth vector field $\vp$ we have $q_x = \nabla v(x)$ a.e. in $G$.

This concludes the proof of Theorem~\ref{dalmasoth}.

\begin{remark} Note that the measure $\Po$ contains the information on the profile of $u$ and all its derivatives, so in the framework of convex integral functionals considered in this section, one could probably deduce from $\Po$ a type of  gradient Young measure at the $\ep $ scale, as defined and characterized in \cite{kp}. Using their characterization  may provide another way of recovering the lower bound of \cite{dalmasopaper}. This is however complicated by the fact that \cite{dalmasopaper} impose  an affine Dirichlet boundary condition rather than  one on the average of the gradient on large cubes. 

 %does not necessarily allow to recover the gradient Young measures of $\{v_\ep\}_\ep$ in the sense of \cite{kp}. The reason is that the Young measures on micropatterns in general record the oscillations of a sequence {\em at a given scale} and therefore erase the oscillations at smaller scales.\footnote{Je ne suis pas sûr que ça réponde à ta footnote. Probablement que l'on peut dépuire de $P$ une gradient young measure à l'échelle $\ep$ qui satisfera les critères de Kinderlehrer-Pedregal, mais je ne vois pas comment en déduire ce qu'on veut. En fait tout serait plus simple et naturel si Dal-maso Modica n'avaient pas choisi des conditions de Dirichlet, mais plutôt une contrainte sur la moyenne du gradient sur $Q_R$.} 
\end{remark}

%\footnote{s'il y a une remarque ici en reference aux gradient young measures de Kinderlehrer et Pedregal, ce serait de dire que $P^\omega$ contient en particulier l'info des profils de $u$ et toutes ses d\'eriv\'ees, donc quand on est dans le cas Dal Maso-Modica, ce que ca encode sur $u$ est le profil de $\nab u$, la limite doit donc satisfaire les proprietes des gradient Young measures que KP identifient (=Jensen pour toutes fonctions quasiconvexes). Peut-etre que ces contraintes impliqueraient une autre maniere de retrouver la borne inf de Dal Maso Modica, en gros retrouver qu'on peut minimiser avec contraintes affines a l'infini, sans utiliser Ziemer}
\section{Application to the two-scale problem of Alberti-Müller}\label{nonconvex}
In this section, we are interested in the functional 
\begin{equation}\label{eam} \eam(v) = \int_0^1 \ep^4 {v''(x)}^2+\frac1{\ep^2} W(v'(x)) + \frac1{\ep^2} m(x) a\(\om,\xe\) v^2(x)\,dx,\end{equation} defined over $H^2([0,1], \mr)$.

This corresponds to  a generalization of the functional $\ep^{-2/3}I^\ep(v)$ defined in \cite{am}, but our $\ep$ corresponds to their $\ep^{1/3}$, and their $a(x)$ is replaced by our $m(x) a(\om, \xe)$, i.e. a randomly oscillating weight at the scale $\ep$ (that is at the scale $\ep^{1/3}$ in the notation of \cite{am}). Here $\om$ is as usual a random parameter belonging to a probability space $\Omega$.

We wish to identify to main order the infimum of $\eam$ on $H^2$ when $\ep\to 0$. We will find out that under suitable assumptions it is a deterministic quantity that can be expressed in terms of a family of sharp-interface problems on the whole real line.

\subsection{$\eam$ in terms of local averages}

We now recast the minimization of $\eam$ in our framework and give precise assumptions.

We let $X = \lp$ and, for any $u\in X$, any $y\in\mr$, any $\ep>0$ and any $\om$ we let $I_y = (y-1/2,y+1/2)$  and 
\begin{equation}\label{fam}\f(u,y) = \begin{cases} +\infty & \text{if $u\notin H^2(I_y)$}\\\displaystyle
\int_{I_y} \ep^2 {u''(t)}^2 +\frac1{\ep^2} W(u'(t)) + m(x) a(\om,t) u^2(t)\,dt & \text{if $u\in H^2(I_y)$.}\end{cases}\end{equation}
Here we have assumed that 
\begin{itemize}
\item[i)] $W(x) = (1-x^2)^2$, although other choices are possible.
\item[ii)]  $x\mapsto m(x)$ is measurable, and $\alpha\le m\le \beta$ for some positive constants $\alpha$ and $\beta$.
\item[iii)] $a$ is measurable and  $a(\om+y,z) = a(\om, y+z)$, for some measure preserving action $(\om, y)\to\om+y$ of $\mr$ on $\Om$. Moreover $1\le a \le 2$
\end{itemize}

\begin{pro} Assume i), ii), and iii) above. Then, as $\ep\to 0$, 
\begin{equation}\label{local}\min_{H^2([0,1])}\eam = \min_{u\in H^2_\loc(\mr)} \E(u) + o(1),\end{equation}
where $\E$ is defined as in \eqref{energy} by 
\begin{equation}\label{enam} \E(u) = \int_0^1 \f\(u,\xe\),dx.\end{equation}
\end{pro}

\begin{proof} In view of \eqref{fam} we rewrite \eqref{enam} as 
$$\E(u) = \int_{x=0}^1\int_{t = \xe-\hal}^{\xe+\hal} \ep^2 {u''(t)}^2 +\frac1{\ep^2} W(u'(t)) + m(x) a(\om,t) u^2(t)\,dt\,dx.$$
Changing variables in the inner integral we find, letting $v(\ep t) = \ep u(t)$, 
$$\E(u) = \int_{x=0}^1\dashint_{y = x-\frac\ep2}^{x+\frac\ep2} \ep^4 {v''(y)}^2 +\frac1{\ep^2} W(v'(y)) + \frac1{\ep^2} m(x) a(\om,y/\ep) v^2(y)\,dy\,dx.$$
Then using Fubini's theorem we find that
$$ \E(u) = \int_\mr \chi_\ep(y)\(\ep^4 {v''(y)}^2 +\frac1{\ep^2} W(v'(y)) + \frac1{\ep^2} m(x) a(\om,y/\ep) v^2(y)\)\,dy,$$
where $\chi_\ep = \frac1\ep\indic_{[-\frac\ep2,\frac\ep2]}* \indic_{[0,1]}$.

Using the fact that $\indic_{[\frac\ep2,1-\frac\ep2]}\le \chi_\ep\le \indic_{[-\frac\ep2,1+\frac\ep2]},$ we easily deduce \eqref{local}.
\end{proof} 
\subsection{Verification of the hypotheses}

We check the hypotheses necessary to apply our framework.

\begin{lem}The action $\theta$ is continuous with respect to $y$ and uniformly continuous with respect to $x$ relatively to $y\in K$, for any bounded $K\subset\mr$.\end{lem}
\begin{proof} If $y_n$ converges to $y$, $u\in X$ and $R>0$  then $u\in L^2(y-R,y+R)$ thus $\theta_{y_n} u\to\theta_{y} u$ in $L^2(-R+1,R-1)$. Since this is true for any $R$, we have $\theta_{y_n} u\to \theta_y u$ in $\lp$.

Now we prove the uniform continuity in $u$. So assume $R>0$ and let $\{y_n\}$ be a sequence in $(-R,R)$. Then if $u_n\to u$ in $\lp$, we have for any $R'>0$ that $u_n\to u$ in $L^2(-(R+R'),R+R')$ therefore
$\theta_{y_n} u_n - \theta_{y_n} u$ tends to $0$ in $L^2(-R',R')$. Since this is true for any $R'$, we obtain the convergence of $\theta_{y_n} u_n - \theta_{y_n} u$ to $0$ in $\lp$ and uniform continuity.
\end{proof}

The stationarity of the functionals is obvious. From the definition of $\theta_y f$ in \eqref{act}, we have, if $u \in H^2(I_z)$
\begin{multline*}\theta_y\f(u,z) = \int_{I_{y+z}} \ep^2 {u''(t-y)}^2 +\frac1{\ep^2} W(u'(t-y)) + m(x) a(\om,t) u^2(t-y)\,dt \\= \int_{I_z} \ep^2 {u''(t)}^2 +\frac1{\ep^2} W(u'(t)) + m(x) a(\om,t+y) u^2(t)\,dt,\end{multline*}
and since $a(\om,t+y) = a(\om+y,t)$, the right-hand side is equal to $\fy(u,z)$.

We also have 
\begin{lem}The functionals $\{\f\}$ are lower semicontinuous, and bounded below by a lower semicontinuous coercive functional $f_0$.\end{lem}
\begin{proof} We begin with the lower semicontinuity. Assume $u_n\to u$ in $\lp$ and $y_n\to y$. If $\liminf_n \f(u_n,y_n) = +\infty$ there is nothing to prove. Otherwise, we consider a subsequence (not relabeled) which realizes the $\liminf$, hence satisfies $\f(u_n,y_n)\to \ell\in\mr_+$. Then any interval $I$ such that $\overline I\subset I_y$ is included in $I_{y_n}$ if $n$ is large enough hence 
$$\limsup_n \int_I \lag(u_n, t)\,dt \le \limsup_n \int_{I_{y_n}} \lag(u_n, t)\,dt= \limsup_n \f(u_n,y_n) = \ell,$$
where $\lag(u,t)$ is the integrand in the integral defining $\f(u,y)$.
It follows that $\{u_n\}$ is bounded in $H^2(I)$ for any such $I$ and then that a subsequence converges weakly in $H^2(I)$ and strongly in $H^1(I)$ by compact embedding, to $u$. Moreover,
$$\int_I e_\ep^{\omega, x}(u,t)\, dt \le \liminf_n \int_{I} \lag(u_n, t)\,dt \le \ell.$$
It follows by taking a sequence $I_k\nearrow I_y$ that $u\in H^2(I_y)$ and that $\f(u,y)\le \ell$.

As a coercive lower semicontinuous functional bounding from below every $\f$ we choose 
$$ f_0(u,y) = \begin{cases} +\infty & \text{if $u\notin H^2(I_y)$}\\\displaystyle
\alpha \int_{I_y}  W(u'(t)) + u^2(t)\,dt & \text{if $u\in H^2(I_y)$.}\end{cases}$$
It is clear that $f_0$ bounds $\f$ from below. The lower semicontinuity of $f_0$ is proven as above. 

To see that $f_0$ is coercive, we assume that $\{C_R\}_{R>0}$ are arbitrary positive numbers and that $\{u_n\}_n$ is such that $\dashint_{B_R} f_0(u_n) \le C_R$ for every $R$. 
It is then straightforward to check that for any $R>0$, $\{u_n\}$ is bounded in $H^1(B_R)$, hence that a subsequence converges in $L^2(B_R)$. Using a diagonal argument we deduce the existence of a subsequence converging in $\lp$, which proves that the set of $u$'s satisfying $\dashint_{B_R} f_0(u_n) \le C_R$ for every $R$ is relatively compact, and the coercivity of $f_0$.
\end{proof}

Now we prove the less trivial two remaining properties: the (M) property, and the uniform measurability.

\begin{pro} The family $\{\f\}$ satisfies the (M) property.\end{pro}
\begin{proof} Assume $\lambda$, $M$ are positive and that $(u_0,y_0)\in \lp\times\mr$. For any $(u,y)\in\lp\times \mr$ we let 
$$v =  \begin{cases} u  & \text{on  $I_y$, }\\
 u_0 & \text{elsewhere.}\end{cases}$$
 Then 
 $$F^\lambda(u,y) := \f(u,y) + \lambda \(d_\lp(u,u_0) + |y-y_0|\) \ge  \f(v,y) + \lambda \(d_\lp(v,u_0) + |y-y_0|\).$$
 This implies that 
 $$R_\lambda \f(u_0,y_0) = \inf_{\lp\times\mr} F^\lambda = \inf_{(u,y)\in\lp\times\mr} F^\lambda(v,y).$$
 Then, $R^\lambda\f\le M$ implies that the last infimum may be taken over the set $K$ of $v$'s and $y$'s such that 
 $$|y-y_0|\le \frac{2M}\lambda,\quad f_0(v,y)\le M,\quad \text{$v=u_0$ outside $I_y$}.$$
 This set is compact: assume $\{y_n\}$ and $\{v_n\}$ satisfy the above bounds, then after extracting a subsequence we have $y_n\to y$. The bound $f_0(v_n,y_n)\le M$ implies that the norm of $v_n$ in $H^1(I_{y_n})$ is bounded, hence the norm of $\theta_{y-y_n} v_n$ is bounded in $H^1(I_y)$. Extracting again, we find that $\theta_{y-y_n} v_n$ converges weakly in $H^1(I_y)$, hence in $L^2(I_y)$. Since  $\theta_{y-y_n} v-n= \theta_{y-y_n}u_0$ outside $I_y$, we obtain the convergence of $\theta_{y-y_n} v_n$ to some $v$ in $\lp$. Since $y_n-y \to 0$  we deduce that $v_n\to v$ in $\lp$. From the lower semicontinuity of $f_0$ it is clear that $(v,y)\in K$, which is therefore compact. This proves property (M) since $K$ is independent of $\ep, x,\om$. 
 \end{proof}
 
\begin{pro} The family $\{\f\}$ is uniformly measurable with respect to $x$.\end{pro}

\begin{proof} First we recall how the distance on the set of functionals $\F$ is defined in Proposition~\ref{compact}, following \cite{dalmaso}. We choose a countable dense subset $\{(u_k,y_k)\}_k \subset \lp\times\mr$, an increasing sequence $\lambda_k$ of positive numbers tending to $+\infty$, and we let for any $f,g\in \F$
$$ d(f,g) = \sum_{i,j\in\mn} \frac1{2^{i+j}} \left|\frac{R_{\lambda_i} f(u_j,y_j)}{1+ R_{\lambda_i} f(u_j,y_j)} - \frac{R_{\lambda_i} g(u_j,y_j)}{1+ R_{\lambda_i} g(u_j,y_j)} \right|.$$

Then, given $\delta>0$, there exists $k\in \mn$ such that $d(f,g)>\delta$ implies that there exists $i,j\le k$ such that 
$$
|R_{\lambda_i} f(u_j,y_j) - R_{\lambda_i} g(u_j,y_j)|>\delta/k,\quad |R_{\lambda_i} f(u_j,y_j)| \le k,\quad |R_{\lambda_i} g(u_j,y_j)|\le k.$$
Assuming without loss of generality that $R_{\lambda_i} f(u_j,y_j) > R_{\lambda_i} g(u_j,y_j)+\delta/k$ and from the definition of $R_\lambda$ in \eqref{yosida}, there exists $(u,y)\in\lp\times\mr$ such that 
\begin{multline*} g(u,y)+\lambda_i \(d_\lp(u,u_j)+|y-y_j|\) \le R_{\lambda_i} g(u_j,y_j)+\frac\delta{2k}\le R_{\lambda_i} f(u_j,y_j) - \frac\delta{2k} \\ \le f(u,y)+\lambda_i \(d_\lp(u,u_j)+|y-y_j|\) - \frac\delta{2k},\end{multline*}
so that in particular 
$$g(u,y) < f(u,y) - \frac\delta{2k},\quad g(u,y)\le k + \frac\delta{2k}, \quad |y|\le \frac1{\lambda_0} \(k + \frac\delta{2k}\) .$$

It follows from the above that   the set $A_{\ep}^h = \{x\in G\mid \exists \om,\, d(\f, f_\ep^{\om, x+h}) >\delta\}$ is included in 
\begin{multline*}B_{\ep}^h= \Big\{x\in G\mid   \text{$\exists (u,z) $, $\exists \om$ s.t. $|\f(u,z) -  f_\ep^{\om, x+h}(u,z)|> \eta$, }\\
\text{$|z|\le \frac1\eta$, and either $|f_\ep^{\om, x+h}(u,z)|<\frac1\eta$ or $|\f(u,z)|<\frac1\eta$.}
\Big\},\end{multline*}
where $\eta>0$ depends only on $\delta$. It remains to prove that $|B_{\ep}^h|$ tends to $0$ as $h\to 0$ uniformly with respect to $\ep$. For this we first note that from \eqref{trans} and the fact that $0<\alpha\le m\le \beta$ and $1\le a\le 2$, it is not difficult to deduce 
$$\text{$|f_\ep^{\om, x+h}(u,z)|<\frac1\eta$ or $|\f(u,z)|<\frac1\eta$}\quad \implies\quad \|(u,z)\|_{H^1(I_z)}\le M,$$
where $M$ is independent  of $\ep$, $y$, which in turn implies an $L^\infty$ bound by (a possibly different) $M$. Assuming this bound, we compute 
\begin{multline*}
|\f(u,z) - f_\ep^{\om, x+h}(u,z)| \le M^2  \int_{I_z} |m(x+h) - m(x)|a(\om, t+y)\,dt\\ \le 2M^2\int_{I_z} |m(x+h) - m(x)|\,dt.\end{multline*}
It follows that 
$$B_{\ep}^h\subset \left\{x\in G\mid \text{$\exists z\in[-1/\eta,\eta]$ s.t.  $\int_{I_z} |m(x+h) - m(x)|\,dt > \frac{\eta}{2M^2}$.}\right\}.$$
The measure of this set tends to $0$ as $h$ tends to $0$, and it is independent of $\ep$. This proves that $|A_{\ep}^h|\to 0$ as $h\to 0$ uniformly with respect to $\ep$, hence the uniform measurability  of $\{\f\}$.
\end{proof}

\begin{pro} We have $\f\overset{\Gamma}{\to}\fb$ uniformly w.r.t. $x,\om$, where 
\begin{equation}\label{bv} \fb(u,y) = \begin{cases} +\infty & \text{if $u'\notin \bv(I_y,\pm 1)$}\\\displaystyle
A_0\|u'\|_{\bv(I_y)} + \int_{I_y} m(x) a(\om,t) u^2(t)\,dt & \text{if $u'\in \bv(I_y,\pm 1)$.}\end{cases}\end{equation}
Here $ \bv(I_y,\pm 1)$ is the space of functions of bounded variation with values in $\{-1,+1\}$, and $A_0 = 2\int_{-1}^1\sqrt W$. 
\end{pro}
\begin{proof}
As in \cite{am}, we will use the following well known result of Modica-Mortola \cite{mm}: 
$$ F_\ep(v) :=  \int_I \ep^2 {v'}^2 + \frac1{\ep^2} W(v)\quad \xrightarrow{\Gamma}\begin{cases} + \infty & \text{if $v'\notin \bv(I, \pm1)$}\\
 A_0 \|v\|_{\bv(I)} & \text{if $v'\in  \bv(I, \pm 1)$}\end{cases},$$
for any open interval $I$, on the space $L^1(I)$.

It is straightforward to deduce that, on $\lp$,  
$$g_\ep(u,y) := \begin{cases} +\infty & \text{if $u\notin H^2(I_y)$}\\\displaystyle
\int_{I_y} \ep^2 {u''(t)}^2 +\frac1{\ep^2} W(u'(t))\,dt & \text{if $u\in H^2(I_y)$}\end{cases}$$
$\Gamma$-converges to 
$$g(u,y) =  \begin{cases} +\infty & \text{if $u'\notin \bv(I_y,\pm 1)$}\\\displaystyle
A_0\|u'\|_{\bv(I_y)} & \text{if $u'\in \bv(I_y,\pm 1)$.}\end{cases}$$
Indeed, assume $(\ue, y_\ep)$ converges to $(u,y)$ and that $\lim_\ep g_\ep(\ue,y_\ep) = \ell\in\mr$. Then for any interval $I\Subset I_y$ the sequence $\{\ue\}$ is bounded in $H^2(I)$, hence converges after extraction in $W^{1,1}(I)$ by compact embedding, thus the derivatives of $\ue$ converge in $L^1(I)$. Then the result of Modica-Mortola implies that $\liminf g_\ep(u_\ep,y_\ep)\ge A_0\|u'\|_{\bv(I_y)}$, and since this is true for any $I\Subset I_y$ we obtain the lower bound part of the desired $\Gamma$-convergence statement. The upper-bound part is straightforward. 

We deduce that $g_\ep$ $\Gamma$-converges to $g$ on $\lp$, and since $g_\ep$ is independent of $x,\om$ the convergence is uniform w.r.t. these variables.

Now, for any $(u,y)\in\lp$,  let 
$$ h^{\om,x}(u,y) = \int_{I_y} m(x) a(\om,t)\,dt.$$
this defines  a functional which is continuous, hence lower semicontinuous, and  independent of $\ep$. Hence $h^{\om,x}$ $\Gamma$-converges to itself  uniformly with respect to $x,\om$. 

Finally, we conclude that $\f = g_\ep + h^{\om,x}$ $\Gamma$-converges as $\ep\to 0$ to $\fb = g + h^{\om,x}$ uniformly with respect to $x$, $\om$.
\end{proof}

\subsection{Lower and upper bounds}
From the results of the preceding section, the abstract framework can be applied and yields
\begin{pro}\label{lowam} Define $\f$ as in \eqref{fam}, with assumptions i), ii) and iii) there satisfied, and define $\Eo$ by \eqref{enam}. Then, assuming the action $(\om,y)\to \om+y$ is ergodic, for almost every $\om_0$, the following holds. 

Assume that $\{\ue\}_\ep$ is a family  in $\lp$ such that $\Eo(\ue) < C$. Then 
\begin{equation}\label{lbam}\liminf_{\ep\to 0} \Eo(\ue) \ge \int_G \alpha_{m(x)}\,dx,\end{equation}
where $\alpha_m$ is defined as \begin{equation}\label{alpham}\alpha_m =  \inf_{\substack{u\in\lp\\u'\in\bv_\loc(\mr,\pm 1)}}\(\limsup_{R\to+\infty} \frac1R \(A_0 \|u'\|_{\bv(-R/2,R/2)}+m \int_{-R/2}^{R/2} a(\om,t) u^2(t)\,dt\)\),\end{equation}where the r.h.s is a.e. independent of $\omega$.
\end{pro}

\begin{proof} From Proposition~\ref{th1} and Theorem~\ref{th2} we have for almost every $\om_0$ 
$$\liminf_{\ep\to 0} \Eo(\ue) \ge \int_0^1\int\( \inf_{u\in\lp}\(\limsup_{R\to+\infty} \dashint_{-R/2}^{R/2} \fb(u,y)\,dy\)\)\,d\om\,dx,$$
which in view of \eqref{bv} is precisely \eqref{lbam}.
\end{proof}

In the remainder of the paper we complement this with an upper bound to obtain

\begin{theo} \label{minam} With the assumptions and notations of the previous theorem, we have for almost every $\om$ the following expansion for the minimum of the energy $\eam$ defined in \eqref{eam} as $\ep\to 0$:
\begin{equation}\label{fin}\min_{H^2([0,1])}\eam = \int_0^1 \alpha_{m(x)}\,dx + o(1).\end{equation}
\end{theo}

\begin{proof} Only the upper bound needs to be proved. This is done by constructing a test-function for the sharp-interface energy. Let  
\begin{equation}\label{X} X = \{v:\mr\to \mr\mid v'\in \bv_\loc(\mr,\pm1)\}\end{equation}
be the space of so-called saw-tooth functions on $\mr$. We prove below that 
for any $\delta>0$ and almost every $\om\in\Om$, there exists for any $\ep>0$ some $\ve\in X$ such that 
\begin{equation}\label{lowdelta} \FH(\ve):=\ep A_0 \|\ve'\|_{\bv([0,1])} + \frac1{\ep^2}\int_0^1 m(x) a\(\om,\frac x\ep\) \ve^2(x)\,dx \le \int_0^1 \alpha_{m(x)}\,dx +\delta.\end{equation}
Moreover, one can choose $\ve\in X$ such that the spacing between two successive jumps  in the derivative of $\ve$ is bounded below by $M \ep$ for some $M>0$ which is independent of $\ep$.

Before proving this fact, we note that it is then straightforward  to derive a corresponding upper-bound for the soft-interface energy $\eam$: If $\ve'$ experiences a jump from $-1$ to $+1$ at $x_0$, say, then we let, for $y\in[-1/\ep,1/\ep]$, 
$$\tve(x_0+\ep^3 y) = \int_0^y\tanh(t/\sqrt 2)\, dt,$$
and glue these transitions so that $\tve'$ is almost constant on the remainder of $[0,1]$. We omit details but it is straightforward to check  that, as $\ep\to 0$, 
$$ \eam(\tve)\le \FH(\ve) + o(1).$$
Therefore for any $\delta>0$ and almost every $\om\in\Om$, there exists for any $\ep>0$ small enough some $\tve$ such that 
$$\eam(\tve) \le \int_0^1 \alpha_{m(x)}\,dx +\delta,$$
proving that \eqref{fin} holds for a.e. $\om$.

It remains to construct for a.e. $\om$, any $\delta>0$ and any $\ep>0$ small enough some $\ve\in X$ satisfying \eqref{lowdelta} and such that the spacing between two successive jumps  in the derivative of $\ve$ is bounded below by $M\ep$ for some $M>0$ which is independent of $\ep$.\medskip

{\em Step 1:}  As a first step,  given  $\delta>0$ we may choose  $\eta>0$ small enough so that $\FH(\ve)$ and $\int_0^1 \alpha_{m(x)}\,dx$ change by at most $\delta/3$ if we replace $m(x)$  by a function  $\mt(x)$, such that  $\|\mt - m\|_\infty < \eta $. This is clearly possible, we omit the proof of this fact. Then we choose an integer $k$ large enough so that the oscillation of $m$ on an interval of size $2/k$ is at most $\eta$, and we let $m_i = m(x_i)$, where $x_i = (i-1)/k + 1/2k$ for any $1\le i \le k$. \medskip

{\em Step 2:} Given $\om$, $R$, $m$ and $u\in X$ we define 
\begin{equation}\label{form}\form(u) = \frac1R \(A_0 \|u'\|_{\bv(-R/2,R/2)}+m \int_{-R/2}^{R/2} a(\om,t) u^2(t)\,dt\).\end{equation}
Using  \eqref{alpham}, for any $m$ and almost every $\om$ 
$$\limsup_{R\to+\infty} \(\inf_{u\in X}  \form(u) \) \le \inf_{u\in X} \(\limsup_{R\to+\infty}  \form(u) \)\le  \alpha_m.$$
Then for any $m$ there exists a set of $\om$'s of measure arbitrarily close to $1$ such that the limit in the left-hand side is uniform w.r.t to $\om$ belonging to this set. Applying this property to $m_1,\dots m_k$ there exists $\Om_\delta\subset \Om$ such that 
\begin{equation}\label{omdelta} \text{$\displaystyle \limsup_{R\to+\infty} \(\inf_{u\in X}  \formi(u) \)$ is uniform w.r.t $\om\in \Om_\delta$ for any $i$},\quad |\Om_\delta|> 1 - \frac\delta{k}.\end{equation}\medskip

{\em Step 3:} Next we use the ergodicity of the action $(\om,x)\to \om +x$ to find that for a.e. $\om$ it holds that 
$$\lim_{R\to +\infty} \frac{\left|\{y\in[0,R]\mid \om+y\in\Om_\delta\}\right|}R = |\Om_\delta|.$$
Applying this to $R = 1/\ep$ we find that for a.e. $\om$, if $\ep>0$ is small enough then for $i=1,\dots,k$ there exists $\txi$ such that
\begin{equation}\label{txi}\om + \frac\txi\ep\in\Om_\delta,\quad |\txi-x_i|<\frac\delta k.\end{equation}\medskip

{\em Step 4:} The building blocks of our construction are now available. For almost every $\om$, given $\ep>0$ small enough we have points $\txi$ satisfying \eqref{txi}. Then applying \eqref{omdelta} we may take $\ep>0$ smaller if necessary such that the minimizer $u_{\ep,i}$ of $f^{\om+\frac\txi\ep}_{\frac1{k\ep},i}(u)$ is such that 
$$f^{\om+\frac\txi\ep}_{\frac1{k\ep},i}(u_{\ep,i})\le \alpha_{m_i} + \frac\delta k.$$
The last step is to glue these pieces together to get a test function. This requires the following lemma and corollary.

\begin{lem}\label{uM} Assume $u$ is a minimizer of $\form$ on $X$, with $R>1$ and $0<\alpha\le m\le \beta.$ then there is a constant depending only on $\alpha$, $\beta$ such that $|u|\le M$ on $[-R/2,R/2]$.
\end{lem}
\begin{proof} Let $t_0\in[-R,R]$ be a point where $|u|$ achieves its maximum $A$, which we assume to be at least $1$.  We assume $u(t_0)$ to be positive and let $[s,s']$ be the connected component of $t_0$ in the set $\{u>A-1\}$. Then we define $\tilde u$ to be equal to $u$ outside $[s,s']$ and equal to $(A-1) - (u+1-A)_+$ on $[s,s'].$ It is straightforward to check that the $\bv$ norm of $\tilde u$ is no greater than that of $u$ plus $2$, and that 
$$\int_s^{s'} m a(\om,t){\tilde u(t)}^2 \,dt \le \int_s^{s'} m a(\om,t){u(t)}^2 \,dt - \alpha\frac{A-1}2,$$
where we have used the fact that $a(\om,t)\ge 1$.

The minimality of $u$ then implies that $A$ is bounded by a constant depending only on $\alpha$.
\end{proof}

\begin{coro} Let $X_R$ be the set of $u\in X$ such that $u(\pm R/2) = 0$. Then assuming $0<\alpha\le m\le \beta$ and $R>1$ we have $\inf_{u\in X_R}\form(u)\le \inf_{u\in X} \form(u) + C/R$, where $C$ depends only on $\alpha,\beta$.
\end{coro}

\begin{proof} The conclusion is trivial if $R<M(\alpha,\beta)$ since in this case the function $u(x) = |x| - R/2$ provides a bound for $\inf_{u\in X_R}\form(u)$. Thus we may assume $R>M$ for any $M$ depending only on $\alpha,\beta$.

Let $u$ be a minimizer of $\form(u)$ on $X$.  We modify $u$ so that $u(\pm R/2) = 0$, and we choose to focus on $u(-R/2) = 0$ alone. The idea is to replace it by an affine function of slope  $1$ which vanishes on the endpoint, for points which are on the left of the crossing between the graph of $u$ and that of the affine function.  
If $u(-R/2) = 0$, then we are done. Otherwise let $$f(t) = \frac{|u(-R/2+t)|}t.$$
Then $f(t)$ tends to $+\infty$ as $t\to 0^+$ and since from Lemma \ref{uM}, $|u|$ is bounded by $M$ we have $f(M)\le 1$. Therefore there exists $t\in[0,M]$ such that $f(t) = 1$ and we may modify $u$ by letting $u(-R/2+ s) = s u(-R/2+t)/t$ if $s \in [0, t]$ (and leaving $u$ unchanged otherwise)  to obtain a function whose derivative belongs to $\{\pm 1\}$ and is zero at $-R/2$. Moreover, since the modification happens on an interval of length bounded by $M$,  one may easily check that we have increased $\form(u)$ by at most a constant depending only on $\alpha,\beta$ in this process, which proves the corollary.
\end{proof}\medskip 

{\em Step 5:}  We may now glue together the functions $u_{\ep,i}$ of the preceding step, or rather we replace $u_{\ep,i}$ by the minimizer $\tuei$ of $f^{\om+\frac\txi\ep}_{\frac1{k\ep},i}$ on $X_{\frac1{k\ep}}$.  Using the corollary above  we have 
$$f^{\om+\frac\txi\ep}_{\frac1{k\ep},i}(\tuei)\le \alpha_{m_i} + \frac\delta k + k\ep M.$$
These functions are equal to zero for $t = \pm 1/2k\ep$ therefore we may define  a test map $\ve\in X$ as follows.  On each of the intervals $[\txi - 1/2k, \txi +1/2k]$ we let 
\begin{equation}\label{bb}\ve(x) =\ep  \tuei\(\frac{x- \txi}\ep\).\end{equation}
The resulting function is defined on $[\tilde x_1 - 1/2k,\tilde x_k + 1/2k]$, which may differ slightly from $[0,1]$ because $\txi - x_i$ may differ by an amount $\delta/k$. If a piece of $[0,1]$ is missing, its size is at most $\delta/k$ and we may define $\ve$ to be a standard sawtooth function of period $\ep$ there.\medskip

{\em Step 6:}  We may now estimate the energy of $\ve$, or rather 
$$\Et(\ve) := \ep A_0 \|v'\|_{\bv([0,1])} + \frac1{\ep^2}\int_0^1 \tilde m(x) a\(\om,\frac x\ep\) v^2(x)\,dx.$$
Indeed from Step~1, it suffices to prove \eqref{lowdelta} to show that 
\begin{equation}\label{lowtilde} \Et(\ve)\le \int_0^1 \alpha_{\tilde m(x)}\,dx + 2\delta.\end{equation}
There are several terms which add up in $\Et(\ve)$. First there is the energy of each of the pieces $\tuei$, for $i = 1\dots k$. Changing variables $t = (x-\tilde x_i)/\ep$ we have 
\begin{multline*}\ep A_0 \|{\ve}'\|_{\bv([\txi - \frac1{2k},\txi + \frac1{2k})]} + \frac1{\ep^2}\int_{\txi - 1/2k}^{\txi + 1/2k} \tilde m(x) a\(\om,\frac x\ep\) {\ve}^2(x)\,dx
 = \ep A_0 \|{\tuei}'\|_{\bv([ -\frac1{2k\ep}, \frac1{2k\ep}])}  + \\ +\ep \int_{- 1/2k\ep}^{ 1/2k\ep} m_i  a\(\om+\frac\txi\ep,t\) {\tuei(t)}^2\,dt = \frac1k f^{\om+\frac\txi\ep}_{\frac1{k\ep},i}(\tuei)\le\frac1k  \alpha_{m_i} + \frac\delta{k^2} + \ep M.\end{multline*}
 Then there are the jumps in the derivative at each of the point $\txi\pm1/2k$, which account for a term bounded by $2 \ep A_0 k$. Finally there is the energy of the standard sawtooth function of period $\ep$ on $[0,1]\setminus [\tilde x_1 - 1/2k,\tilde x_k + 1/2k]$, which is bounded by $C\delta/k$.
 
 We deduce that 
 $$\Et(\ve)\le \frac1k  \sum_{i=1}^k \alpha_{m_i} + \frac\delta{k} + k \ep M + C\frac\delta k,$$
 which proves \eqref{lowtilde} and then \eqref{lowdelta}.\medskip
 
 {\em Step 7:}  It remains to prove that the space between successive jumps in the derivative of $\ve$ are bounded below by  $\eta\ep$ for some $\eta>0$. In view of our construction, this amounts  to show that this holds for our building blocks defined by \eqref{bb}, and then to prove that the spacing between two successive jumps in the the derivative of a minimizer of $\form$ is bounded below by a constant $\eta>0$ independent of $R>1$, $m$ satisfying $0<\alpha\le m\le \beta$, and $\om$. 

Let $u$ be such a minimizer (with the continuous line graph). Let $x_0$, $x_0+\eta$ be  two consecutive jumps in the derivative of $u$, assuming $u'=+1$ on $(x_0,x_0+\eta)$. We define the competitor  $\tu$ (with the dashed-line graph) to be equal to $u$ for $x<x_0$ and to be equal to $u(x+\eta) - \eta$ for $x>x_0$. Clearly, if $x_1$ is the jump following $x_0+\eta$ (or $x_1 = R$ if $x_0+\eta$ is the last jump) then $\tu(x) = u(x)$ for $x>x_1$. Moreover,  using Lemma~\ref{uM} we find that $x_1 - x_0 \le 2M$ and that  
$$\int_{x_0}^{x_1} m a(\om,t) \tu^2(t)\,dt \le \int_{x_0}^{x_1} m a(\om,t) u^2(t)\,dt + C M^2 \eta.$$
Since $\tu$ has one less jump than $u$ we deduce from the minimality of $u$ that 
$$0\le R(\form(\tu) - \form(u))\le C M^2 \eta - A_0,$$
from which we find a lower bound for $\eta$ as desired.

This finishes the proof that $\ve$ satisfies \eqref{lowtilde} and has minimal spacing between jumps bounded below by $\eta\ep$, and the proof of the theorem.
\end{proof}
{\bf Acknowledgments:} The authors wish to thank the referee for his/her careful reading of the manuscript and comments. They also wish to thank V.Bergelson and E.Lesigne for pointing them to references \cite{nz} and \cite{krengel}. The work of L.B. and E.S. was partially supported by NSF grants DMS-1405769 and DMS-1106666.

\end{document}